\documentclass[11pt]{article}

\usepackage{color,graphicx}
\usepackage{caption}
\usepackage{subcaption}

\usepackage{theorem}
\usepackage{amssymb, amsmath, amsbsy, amsfonts}

\newtheorem{theorem}{Theorem}[section]
\newtheorem{lemma}[theorem]{Lemma}
\newtheorem{proposition}[theorem]{Proposition}

\newtheorem{definition}[theorem]{Definition}

\newcommand{\R}{\mathbb{R}}

\newcommand{\dive}{{{\rm div}\;}}
\newcommand{\sH}{{\rm H}}

\newcommand{\sL}{{\rm L}}

\newcommand{\dif}{{\rm d}}
\catcode`\@=11
\def\theequation{\@arabic{\c@section}.\@arabic{\c@equation}}    
\def\thetheorem{\@arabic{\c@section}.\@arabic{\c@theorem}}
\def\theenumi{\@roman{\c@enumi}}
\catcode`\@=12

\def\blacksquare{
\thinspace\nobreak \vrule width 5pt height 5pt depth 0pt}

\newenvironment{proof}{\begin{trivlist}
                       \item[]\hspace{0cm}{\bf Proof: }
                       \hspace{0cm} }{\hfill $\blacksquare$
                     \end{trivlist}}
\newenvironment{proofof}[1]{\begin{trivlist}
                       \item[]\hspace{0cm}{\bf Proof of #1: }
                       \hspace{0cm} }{\hfill $\blacksquare$
                     \end{trivlist}}

\numberwithin{figure}{section}

\begin{document}
\title{Minimization of the ground state of the mixture of two conducting materials in a small contrast regime}
\author{C. Conca, M. Dambrine, R. Mahadevan and D. Quintero}
\date{}
\maketitle

\textbf{Abstract.} We consider the problem of distributing two conducting materials with a prescribed volume ratio in a given domain so as to minimize the first eigenvalue of an elliptic operator with Dirichlet conditions. The gap between the two conductivities is assumed to be small (low contrast regime). For any geometrical configuration of the mixture, we provide a complete asymptotic expansion of the first eigenvalue. We then consider a relaxation approach to minimize the second order approximation with respect to the mixture. We present numerical simulations in dimensions two and three.

\section{Introduction}

Problems of mininimizing the ground state of composite materials appear frequently and are of interest in applications. We refer to Henrot \cite{Henrot},  Cox and McLaughlin \cite{CoxMcLaughlin1, CoxMcLaughlin2}, Cox and Lipton\cite{CoxLipton} and included references. In this article, we consider the following problem. Given a domain $\Omega$ and a subdomain $B$ and two nonnegative numbers $\alpha$ and $\beta$, we define the ground state $\lambda(B)$ of the mixture as the infimum of the $\lambda$ such that there exists $0 \ne u$ such that
\begin{equation}\label{evp-basic}
-\dive{\big((\alpha+(\beta-\alpha)\chi_{B})\nabla u \big)}=\lambda u \text{ in }\Omega \text{ and }u=0 \text{ on }\partial\Omega\,.
\end{equation}
In other words, $\lambda(B)$ is the smallest eigenvalue of the operator $-\dive{\big( (\alpha+(\beta-\alpha)\chi_{B})\nabla .\big)}$ on $\sH^1_{0}(\Omega)$. We are then interested in minimizing $\lambda(B)$ with respect to $B$  among the subdomains of $\Omega$ of given volume.

In general, it is well-known that this problem is not wellposed: the infimum is not usually reached at a given $B$ and we have to consider a relaxed version corresponding to a situation of homogenization (see \cite{CoxLipton}). 

Nevertheless, when $\Omega$ is a ball, the infimum is reached on a radially symmetric domain $B^*$ (see \cite{LionsTrombetti},\cite{ConcaMahadevan2}). In the recent years, much attention has been put on the determination of the corresponding $B^*$. First, Conca and al.\ conjectured in \cite{ConcaMahadevan1} that the global minimizer $B^*$ in $\Omega$ should be a concentric  ball of the prescribed volume. The conjecture was motivated by the situation in dimension one and by numerical simulations. Then, Dambrine and Kateb reinforced the conjecture by an order two sensitivity analysis in \cite{DambrineKateb} by proving that the concentric ball of prescribed volume is a local strict minimizer of $\lambda(B)$.

However, Conca et. al.\ proved in \cite{ConcaLaurainMahadevan} that the conjecture is false. Their strategy was the following. They consider the case of small contrast, that is to say, $\alpha$ and $\beta$ such that the difference of both conductivities is small: $\beta=\alpha(1+\varepsilon)$ and provide the first order asymptotic expansion $\lambda_1(B)$ of $\lambda(B)$   with respect to the small parameter $\varepsilon$ for any admissible domain $B \subset \Omega$. Then, they minimize the new objective functional $\lambda_{1}(B)$ with respect to $B$ and observe that the minimizer $B_1$ of this approximation is not always the concentric ball of prescribed volume. Finally, thanks to a precise estimate of the remainder in the approximation,  they prove that $\lambda(B_{1}) <\lambda(B^{*})$. 

Finally, Laurain proved in \cite{Laurain2014} that the global minimum of the first eigenvalue in low contrast regime is either a centered ball or the union of a centered ball and of a centered ring touching the boundary, depending on the prescribed volume ratio between the two materials. Thus the small contrast case is well understood when the domain is  a ball. 

We aim in this work to make a precise analysis of the small contrast case in general domains.  In Section 2, to begin with, we characterize completely the full asymptotic expansion of $\lambda(B)$ with respect to the small parameter $\varepsilon$. Subsequently, we obtain a second order approximation $\lambda_{2}(B)$ of $\lambda(B)$ with uniform estimates for the remainder, uniform with respect to $B$. This means that minimizers for the second order approximation $\lambda_2(B)$ are approximate minimizers for the original objective functional $\lambda(B)$. With this motivation, in Section 3, we study the problem of minimizing $\lambda_{2}$. Unlike the first order approximation $\lambda_1(B)$, the minimization problem for $\lambda_2(B)$ is not,  {\em a priori}, well posed and thus, qualitatively, resembles more closely the minimization problem for $\lambda(B)$. A relaxed formulation for the minimization problem for $\lambda_2(B)$ is obtained using $H$-measures. It can be seem that the relaxed problem for $\lambda_2(B)$ has a much more simple aspect compared to the relaxed problem for $\lambda(B)$ obtained in Cox and Lipton \cite{CoxLipton}. Finally, in Section 4, the optimality conditions for the relaxed problem for $\lambda_2(B)$ are obtained and the minimization problem is studied numerically using a descent algorithm. 

%%%%%%%%%%%%%%%%%%%%%%%%%%%%%%%%%%%%%
%%%%%  \input{sec1_r}
%%%%%%%%%%%%%%%%%%%%%%%%%%%%%%%%%%%%%
\section{Asymptotic expansion of the first eigenvalue with respect to the constrast.}
\setcounter{equation}{0}
\setcounter{theorem}{0}

We consider the low contrast regime, that is to say, $\alpha$ and $\beta$ such that the difference of both conductivities is small: $\beta=\alpha(1+\varepsilon)$. We shall denote the first eigenvalue in the problem \eqref{evp-basic} by $\lambda_\varepsilon(B)$ for a given distribution $B$ of the material with conductivity $\beta$ and a given value of the contrast parameter $\varepsilon>0$. 

The existence of an asymptotic  development for $\lambda_\varepsilon(B)$, for given $B$, is classical from perturbation theory of simple eigenvalues. By the  Krein-Rutman theorem, the first eigenvalue $\lambda_\varepsilon(B)$ in \eqref{evp-basic} is simple. The corresponding normalized eigenfunction, with unit $L^2$ norm and taken to be non-negative, will be denoted by  $u_\varepsilon(B)$. So, by classical results from perturbation theory (see, for instance, Theorem 3, Chapter 2.5 of Rellich \cite{Rellich} ), for a given $B$, the map $\varepsilon \mapsto\left( \lambda_{\varepsilon},u_{\varepsilon}\right)$ is analytic in $(\R,\sH^1_{0}(\Omega))$. Therefore there are sequences $(\lambda_{i})$ of real numbers and $(u_{i})$ of functions in $\sH^1_{0}(\Omega)$ such that:
\begin{equation}
\lambda_{\varepsilon}=\sum_{i=0}^{\infty} \lambda_{i}\varepsilon^{i}  \text{ and }
u_{\varepsilon}= \sum_{i=0}^{\infty} u_{i}\varepsilon^{i}. \label{developpement:vp:fp}
\end{equation}
As a consequence, there are constants $C_{n}(B)$ such that
\begin{equation*}
\left| \lambda_{\varepsilon}-\sum_{i=0}^{n} \lambda_{i}\varepsilon^{i} \right| \leq C_{n}(B) \varepsilon^{n+1} \text{ and }
\left\| u_{\varepsilon}- \sum_{i=0}^{n} u_{i}\varepsilon^{i} \right\|_{\sH^1_{0}} \leq C_{n}(B) \varepsilon^{n+1}.
\end{equation*}
In this section, we will first identify the coefficients $\lambda_{i}$, $u_{i}$ then prove that the constants $C_{n}(B)$ can be taken uniform in $B$. This will serve in obtaining an approximate model problem for the eigenvalue minimization problem. 

\subsection{Computation of the coefficients in \eqref{developpement:vp:fp}}

The terms in the the asymptotic expansions in \eqref{developpement:vp:fp} may be identified, formally, by injecting the expansions in the equations defining  $(\lambda_{\varepsilon},u_{\varepsilon})$, that is, 
\begin{eqnarray*}
-\dive{\left(\alpha(1+\chi_{B}\varepsilon) \nabla \left(\sum_{i=0}^{\infty} u_{i}\varepsilon^{i}\right)\right)}&=&\left( \sum_{i=0}^{\infty} \lambda_{i}\varepsilon^{i}\right) \left(\sum_{i=0}^{\infty} u_{i}\varepsilon^{i}\right) \mbox{ in }\Omega,\\
\sum_{i=0}^{\infty} u_{i}\varepsilon^{i} & =&0 \mbox{ on } \partial \Omega,\\
\int_{\Omega}\left( \sum_{i=0}^{\infty} u_{i}\varepsilon^{i} \right)^2 &=&1.
\end{eqnarray*}
and we obtain then the following relationships by identifying the coefficients of same order in the previous power series.
\begin{equation}\label{order0}
\left\{\begin{array}{rcl}
-\alpha\Delta u_{0}-\lambda_{0}u_{0}&=&0 \text{ in }\Omega\\
u_{0}&=&0 \text{ on }\partial\Omega \quad \forall i\geq0,\\
\displaystyle\int_{\Omega}u_0^2 &=& 1\,.
\end{array}\right.
\end{equation}
\begin{equation}\label{order-i}
\left\{\begin{array}{rcl}
-\alpha \Delta u_{i} -\lambda_{0} u_{i} &=& \dive{\left(\alpha \chi_{B} \nabla u_{i-1}\right)}+\displaystyle\sum_{k=1}^i \lambda_{k}u_{i-k} \text{ in }\Omega \quad \forall i \geq1, \\
u_{i}&=&0 \text{ on }\partial\Omega \quad \forall i\geq0,\\
\displaystyle\sum_{k=0}^{i} \int_{\Omega} u_{k}u_{i-k}&=&0 \quad \forall i\geq1.
\end{array}\right.
\end{equation}
It is possible to rigorously justify the relations by using the expansions \eqref{developpement:vp:fp} in the weak formulation of the partial differential equation in \eqref{evp-basic}.  We then have an iterative procedure to compute the pair $(\lambda_{i},u_{i})$. 
\paragraph{The case: $i=0$.} By definition, one has:
\begin{eqnarray}
-\alpha\Delta u_{0}-\lambda_{0}u_{0}&=&0 \text{ in }\Omega \label{ecuacionu0}\\
u_{0}&=&0 \text{ on }\partial\Omega.
\end{eqnarray}
Hence, the couple $(\lambda_{0},u_{0})$ is an eigenpair of $-\alpha\Delta$ with homogeneous Dirichlet boundary condition. Clearly $u_{0} \geq 0$ in $\Omega$ since $u_{\varepsilon} \to u_0$ as $\varepsilon \rightarrow 0$ and the eigenmodes $u_\varepsilon$ are non-negative. Now, by the Krein-Rutman theorem, since all eigenmodes change sign except those associated to the first eigenvalue, we obtain that $\lambda_{0}$ is the ground state of $-\alpha\Delta$ with Dirichlet boundary condition and $u_{0}$ is the positive eigenmode with $\sL^2$-norm $1$.

Now assume that, for a given $i$, we have knowledge of all the $\lambda_{k},u_{k}$ for $k < i$.  We now then treat 

\noindent
{\bf The case $k=i$.} We know that $u_i$ satisfies the equation
\begin{eqnarray}
\label{eq5}
-\alpha \Delta u_{i} -\lambda_{0} u_{i} &=& \dive{\left(\alpha \chi_{B} \nabla u_{i-1}\right)}+\sum_{k=1}^i \lambda_{k}u_{i-k} \mbox{ in } \Omega,  \\
u_{i}&=&0 \text{ on }\partial\Omega,\nonumber
\end{eqnarray}
Notice that the right hand side has the unknown quantity $\lambda_i$. We shall first obtain an expression for $\lambda_i$ in terms of $\lambda_k$'s and $u_k$'s for $k<i$ which have been assumed to be calculated previously.  The compatibility condition, the Fredholm alternative for the equation \eqref{eq5},  imposes the orthogonality of the right hand side of the former equation to the kernel of  $-\alpha\Delta -\lambda_0 I$ with Dirichlet boundary condition  which is spanned by $u_{0}$
\begin{equation*}
\int_{\Omega} \left( \dive{\left(\alpha \chi_{B} \nabla u_{i-1}\right)}+\sum_{k=1}^i \lambda_{k}u_{i-k} \right) u_{0}=0.
\end{equation*}
This gives the expression for the eigenvalue $\lambda_{i}$
\begin{equation}\label{eqlandai}
\lambda_{i}= \int_{B}\alpha \nabla u_{i-1}\cdot\nabla u_{0} - \sum_{k=2}^{i-1} \int_{\Omega} \lambda_{i-k}u_{0}u_{k}
\end{equation}
taking into account the fact that the $L^2$ norm of $u_0$ is $1$ and, $u_0$ and $u_1$ are orthogonal.  In the sequel, whenever there is a sum whose upper limit is less than the lower limit, we shall adopt the convention that the sum is $0$. 

Now, to end, we note that $u_i$ is not completely determined by the equation (\ref{eq5}), but only upto the kernel of $-\alpha\Delta -\lambda_0 I$. For $i=0$, the non-negativity of $u_0$ and the normalization condition (the third relation in \eqref{order0}) determines uniquely $u_0$. For general $i$, having determined  uniquely the $u_k$ for $k<i$, the term $u_i$ is determined uniquely using the normalization condition (the third relation in \eqref{order-i} which can be written as 
\begin{equation}\label{productoesc}
\int_{\Omega} u_{i}u_{0}= -\frac{1}{2}\sum_{k=1}^{i-1} \int_{\Omega} u_{k}u_{i-k}.
\end{equation}
and should be understood as the orthogonality relation $\int_{\Omega} u_{i}u_{0}=0$ when $i=1$.

\subsection{Uniform estimate of the remainders}

We seek to estimate the remainder in the expansions \eqref{developpement:vp:fp}, uniformly in $B$. Our main results in this section are the following estimates.
\begin{proposition}
\label{estimation:reste:1}
There exists a constant $C$, independent of $B$, such that
\begin{equation}\label{reste:ordre1}
|\lambda_{\varepsilon}-(\lambda_{0}+\varepsilon\lambda_{1})|\leq
 \sqrt{\frac{\lambda_0}{\alpha}}C \varepsilon^2.
\end{equation}
\end{proposition}

\begin{proposition}
\label{estimation:reste: 2}
There is a constant $C>0$ independent of $B$ such that:
\begin{equation}\label{reste:ordre2}
 |\lambda_{\varepsilon}-(\lambda_0+\varepsilon\lambda_1+\varepsilon^2\lambda_2)|\leq 2\, C \varepsilon^3\sqrt{\frac {\lambda_0}{\alpha}}.
\end{equation}
\end{proposition}
The main tool we use for the estimation of the remainders is the notion of \emph{$h$-quasimode} with $h=\mathcal{O}(\varepsilon^k)$, for $k=1,2$ in the sequel. The notion of quasimode is defined as follows.
\begin{definition}
Let $A$ be a self-adjoint operator on a Hilbert space $H$ with domain $D(A)$. For a fixed $h>0$, a pair $(\lambda,u)\in \R \times D(A)\setminus\{0\} $ is called a \emph{$h$-quasimode} if we have
$$\| (A-\lambda)u\|_{H} \leq h \|u\|_{H}.$$
\end{definition}
The interest of such a definition relies on the following fact: if $(\lambda,u)$ is a \emph{$h$-quasimode} of $A$, then the distance from $\lambda$ to the spectrum of $A$ is less than $h$ and the distance between $u$ and certain eigenspaces of $A$ can be estimated (See Lemma 2-2 in \cite{DaugeFaou}). We will prove that our truncated power series expansions are quasimodes in the Hilbert space $\sH^{-1}(\Omega)$.
%Remark: the letter $C$ designs a constant independant of $B$ and $C$
\subsubsection*{Remainder of order one.}
The first step is to prove a uniform bound in $B$ of $\|u\|_{\sH^1(\Omega)}$.
\begin{lemma}
 \label{borne:uniforme:u1}
There exists $C$, which is independent of $B$, such that:
\begin{equation}\label{estimacionu1}
\|u_{1}\|_{\sH^1_{0}(\Omega)}\le C \text{ and }|\lambda_{\varepsilon}-\lambda_{0}|\leq C \varepsilon.
\end{equation}
\end{lemma}
\begin{proofof}{Lemma \ref{borne:uniforme:u1}} By using \eqref{eqlandai}, with $i=1$, we have the following expression and uniform bounds for 
$\lambda_{1}(B)$ 
\begin{equation}\label{estimativolanda1}
\lambda_{1} = \int_{B} \alpha |\nabla u_{0}|^2 \leq \alpha \int_{\Omega} |\nabla u_{0}|^2= \lambda_{0}\,.
\end{equation}
By \eqref{order-i}, for $i=1$, $u_{1}$ satisfies the following:
\begin{eqnarray}\label{order1-id}
-\alpha\Delta u_{1}-\lambda_{0}u_{1}& =& -\dive{\left(\alpha\chi_{B}\nabla u_{0}\right)} \text{ in }\Omega,\\
u_{1}&=& 0 \text{ on }\partial\Omega,\\
\int_{\Omega}u_{0}u_{1}&=&0.
\end{eqnarray}
After multiplying the first relation by $u_{1}$  and integrating over $\Omega$, by integration by parts, we get
$$\int_{\Omega}\alpha|\nabla u_{1}|^2-\lambda_{0}\int_{\Omega} u_{1}^2 =\int_{B} \alpha\nabla u_{0}     \cdot \nabla u_{1}.$$
By the characterization of the spectrum of an elliptic self-adjoint operator using the Rayleigh's quotient, we know that for all $v$ in $\sH^1_{0}(\Omega)$ orthogonal to the first eigenfunction $u_{0}$, it holds that 
\begin{equation}\label{vp2}
\lambda^1 \int_{\Omega} v^2 \leq \alpha \int_{\Omega} |\nabla v|^2,
\end{equation}
where $\lambda^1>\lambda_{0}$ is the second eigenvalue of $-\alpha\Delta$ in $\sH^1_{0}(\Omega)$. We have used the superscript here to distinguish the second eigenvalue $\lambda^1$ from $\lambda_1$ which appears in the second term of the expansion  \eqref{developpement:vp:fp}. Since $u_1$ is orthogonal to $u_0$, it follows using \eqref{vp2} that
\begin{equation}\label{estimacion}
\alpha\left( 1-\frac{\lambda_{0}}{\lambda^1} \right) \int_{\Omega}|\nabla u_{1}|^2 \leq \int_{\Omega}\alpha|\nabla u_{1}|^2-\lambda_{0}\int_{\Omega} u_{1}^2  \leq \alpha\| u_{0} \|_{\sH^1_{0}(\Omega)} \|u_{1}\|_{\sH^1_{0}(\Omega)}
\end{equation}
where at the end we have used \eqref{order1-id} and followed it by a simple estimation. We have obtained the upper bound for $u_{1}$. Finally, using the variational characterization of the first eigenvalue for elliptic self-adjoint operators, we obtain 
\begin{align*}
 \lambda_{0} = \int_{\Omega} \alpha|\nabla u_{0}|^2 \leq \int_{\Omega} \alpha|\nabla u_{\varepsilon}|^2  & \leq\int_{\Omega} \alpha(1+\chi_{B}\varepsilon) |\nabla u_{\varepsilon}|^2 =\lambda_{\varepsilon} \\
 &\leq \int_{\Omega} \alpha(1+\chi_{B}\varepsilon) |\nabla u_{0}|^2 \leq (1+\varepsilon) \int_{\Omega} \alpha |\nabla u_{0}|^2 = (1+\varepsilon) \lambda_{0}
\end{align*}
which allows us to conclude that that $|\lambda_{\varepsilon}-\lambda_{0}|\leq C \varepsilon$.
\end{proofof}

\noindent To use the quasimode strategy, we compute:
\begin{align}
-\dive & {\left(\alpha(1+\chi_{B}\varepsilon)  \nabla (u_{0}+\varepsilon u_{1}) \right)} -(\lambda_{0}+\varepsilon\lambda_{1})(u_{0}+\varepsilon u_{1}) \nonumber \\
&= -\alpha \Delta u_{0} -\lambda_{0} u_{0}  + \varepsilon \left( -\alpha\Delta u_{1} - \lambda_{0}u_{1}-\lambda_{1}u_{0}-\dive{(\alpha\chi_{B} \nabla u_{0})} \right) \nonumber \\
&~~~~+ \varepsilon^2 \left(- \lambda_{1} u_{1} -\dive{(\alpha\chi_{B}\nabla u_{1})}\right) \nonumber\\
&=\varepsilon^2 \left( -\lambda_{1} u_{1} -\dive{(\alpha\chi_{B}\nabla u_{1})}\right)\label{desllord1}
\end{align}
where we have used (\ref{ecuacionu0}) and, (\ref{eq5}) with $i=1$.
\begin{proofof}{ Proposition \ref{estimation:reste:1}}
We need a uniform bound on the normalized right-hand side: $\lambda_{1} u_{1} -\dive(\chi_{B}\nabla u_{1})$. Obviously, this term is only defined in $\sH^{-1}(\Omega)$ hence we have to make the estimation in the $\sH^{-1}(\Omega)$ norm. To that end, we use a test function $\varphi   \in \sH^1_{0}(\Omega)$ and compute the duality product:
\begin{align*}
\langle -\dive(\alpha\chi_{B}\nabla u_{1}),\varphi \rangle_{\sH^{-1}(\Omega)\times\sH_{0}^1(\Omega)} &= \int_{\Omega}\alpha\chi_{B} \nabla u_{1} \cdot \nabla \varphi = \int_{B} \alpha\nabla u_{1}\cdot \nabla \varphi \\
& \leq \alpha \|u_{1}\|_{\sH^1_{0}(\Omega)} \|\varphi\|_{\sH^1_{0}(\Omega)}.
\end{align*}
This proves that
\begin{equation}\label{estimaciondiv}
\|-\dive{\alpha(\chi_{B}\nabla u_{1})}\|_{\sH^{-1}(\Omega)}\leq \alpha \|u_{1}\|_{\sH^1_{0}(\Omega)}.
\end{equation}
And
\begin{align*}
\nonumber\langle\lambda_1u_1,\varphi\rangle_{\sH^{-1}(\Omega)\times\sH_{0}^1(\Omega)} &= \int_{\Omega}\lambda_1u_1\varphi\le \lambda_1\|u_1\|_{\sL^2(\Omega)}\|\varphi\|_{\sL^2(\Omega)}\\
\label{estimacion-lambda1u1}&\le \lambda_1 \|u_{1}\|_{\sH^1_{0}(\Omega)}\|\varphi\|_{\sH^1_{0}(\Omega)} \le C \|\varphi\|_{\sH^1_{0}(\Omega)} 
\end{align*}
using the estimation (\ref{estimativolanda1}) and the fact that $u_1$ is bounded independently of $B$ proved in Lemma \ref{borne:uniforme:u1}. This gives
\begin{equation}\label{estimacionlanda1u1}
\|\lambda_1u_1\|_{\sH^{-1}(\Omega)}\le C\,.
\end{equation}
Hence, we obtain from (\ref{desllord1}), using (\ref{estimaciondiv}) and (\ref{estimacionlanda1u1})  that there exists a constant $C$ independent of $B$ such that
\begin{equation}\label{eq12}
\| -\dive{\left(\alpha(1+\chi_{B}\varepsilon) \nabla (u_{0}+\varepsilon u_{1}) \right)} -(\lambda_{0}+\varepsilon\lambda_{1})(u_{0}+\varepsilon u_{1}) \|_{\sH^{-1}(\Omega)} \leq C \varepsilon^2
\end{equation}
Moreover, using $u_{0} \in \sH^1_{0}$ as test function in the definition of the $\sH^{-1}$-norm of $u_{0}+\varepsilon u_{1}$, we obtain
\begin{align}
\|u_0+\varepsilon u_1\|_{\sH^{-1}(\Omega)} &= \sup_{\varphi\in\sH_0^1(\Omega)}\frac{\langle u_0+\varepsilon u_1,\varphi\rangle_{\sH^{-1},\sH_0^1}}{\|\varphi\|_{\sH_0^1(\Omega)}}=\sup_{\varphi\in\sH_0^1(\Omega)}\cfrac{ \displaystyle \int_{\Omega}(u_0+\varepsilon u_1)\varphi}{\|\varphi\|_{\sH_0^1(\Omega)}} \nonumber \\[5pt]
&\ge \cfrac{\displaystyle\int_{\Omega}(u_0+\varepsilon u_1)u_0}{\|u_0\|_{\sH_0^1(\Omega)}}= \cfrac{\displaystyle\int_{\Omega}u_0^2}{\left(\displaystyle \int_{\Omega}|\nabla u_0|^2\right)^{\frac 12}}= \sqrt{\cfrac{\alpha}{\lambda_0}}\label{des13}.
\end{align}
Hence, by \eqref{eq12} and \eqref{des13}, we obtain
$$\| -\dive{\left(\alpha(1+\chi_{B}\varepsilon) \nabla (u_{0}+\varepsilon u_{1}) \right)} -(\lambda_{0}+\varepsilon\lambda_{1})(u_{0}+\varepsilon u_{1}) \|_{\sH^{-1}(\Omega)} \leq \sqrt{\frac{\lambda_0}{\alpha}}\, C \varepsilon^2\,\|u_0+\varepsilon u_1\|_{\sH^{-1}(\Omega)}$$
As a consequence of the theory of quasi mode, there is an element of the spectrum of the self-adjoint operator $-\dive{\left(\alpha(1+\chi_{B}\varepsilon) \nabla \cdot \right)}$ in $\sH^{-1}(\Omega)$ at distance at most $\sqrt{\frac{\lambda_0}{\alpha}}C \varepsilon^2$ from $\lambda_{0}+\varepsilon\lambda_{1}$.  To finish, we need to argue that this element of the spectrum is $\lambda_\varepsilon$, the first eigenvalue of  $-\dive{\left(\alpha(1+\chi_{B}\varepsilon) \nabla \cdot \right)}$. If these were higher eigenvalues, then as $\varepsilon \to 0$, they would tend to a higher eigenvalue of the operator $-\alpha \Delta$. But this would lead to a contradiction, since this sequence is within a distance $O(\varepsilon^2)$ from the sequence $\lambda_0+\varepsilon \lambda_1$ which tends to $\lambda_0$, the first eigenvalue of   $-\alpha \Delta$ which is simple.
\end{proofof}

\subsubsection*{Remainder of order two}

We first prove an uniform upper bound for $\lambda_{2}$ and $u_{2}$.
\begin{lemma}
 \label{borne:uniforme:u2}
There exists $C$, which is independent of $B$, such that:
\begin{equation}\label{estimacionu2}
\|u_{2}\|_{\sH^1_{0}(\Omega)}\le C \text{ and }\lambda_2\leq C.
\end{equation}
\end{lemma}
\begin{proofof}{Lemma \ref{borne:uniforme:u2}}
First, notice that by \eqref{eqlandai} applied with $i=2$, we get
\begin{equation}\label{eqlanda2}
\lambda_2=\int_{B}\alpha\nabla u_{0}\cdot\nabla u_{1}\le\alpha\|u_{0}\|_{\sH^1_{0}(\Omega)}\|u_{1}\|_{\sH^1_{0}(\Omega)}\le C\
\end{equation}
where $C$ is independent of $B$ by the estimate (\ref{estimacionu1}).
In a second step, we search a uniform estimate for $u_{2}$. To that end, we follow the strategy already used to estimate $u_1$. The main change is that $u_2$ is not orthogonal to $u_0$ so the adaptation is not straightforward. To overcome the difficulty we introduce the combination $u_2+au_0$ where
$$a=-\int_{\Omega}u_2u_0$$
 is chosen such that  $u_2+au_0$ is $\sL^2(\Omega)$-orthogonal to $u_0$. \par
By \eqref{productoesc} for $i=2$ we have 
\begin{equation}\label{prodint}
\int_{\Omega}u_2u_0=-\frac 12\int_{\Omega}u_1^2
\end{equation}
which gives
\begin{equation}\label{estimaciona}
a=\frac 12\int_{\Omega}u_1^2\le\frac 12 \|u_1\|_{\sH^1(\Omega)}^2\le C
\end{equation}
with $C$ independent of $B$ (by \eqref{estimacionu1}). We now estimate $u_{2}+au_{0}$. For this, we multiply equation \eqref{ecuacionu0} by $a$ and add it to equation \eqref{eq5} to obtain:
\begin{eqnarray*}
-\alpha\Delta(u_2+au_0)-\lambda_0(u_2+au_0) & =& \dive(\alpha\chi_B\nabla u_1)+\lambda_1u_1+\lambda_2u_0,\quad \text{in}\ \Omega\\
u_2+au_0 & =& 0 \quad\text{on}\ \partial\Omega
\end{eqnarray*}
Using $u_2+au_0$ as test function, it follows that
\begin{align}
\alpha\int_{\Omega}| &\nabla(u_2+au_0)|^2-\lambda_0\int_{\Omega}(u_2+au_0)^2 \nonumber\\
&= \int_B\alpha\nabla u_1\cdot\nabla
(u_2+au_0)+\int_{\Omega}\lambda_1u_1(u_2
+au_0)+\int_{\Omega}\lambda_2u_0(u_2+au_0) \nonumber\\
&\le \Bigl(\alpha\|u_1\|_{\sH_0^1(\Omega)}+\lambda_1\|u_1\|_{\sH_0^1(\Omega)}
+|\lambda_2|\,\|u_0\|_{\sH_0^1(\Omega)}\Bigr)\,\|u_2+au_0\|_{\sH_0^1(\Omega)} \nonumber\\
&\le C\,\|u_2+au_0\|_{\sH_0^1(\Omega)}\label{des18}
\end{align}
where $C$ is independent of $B$, by estimates (\ref{estimativolanda1}), (\ref{estimacionu1}) and (\ref{eqlanda2}). Since $u_2+au_0$ is orthogonal to $u_0$, similarly as in the estimation \eqref{estimacion}, we conclude that $u_2+au_0$ is bounded in $\sH_0^1(\Omega)$ uniformly in  $B$. Therefore,
\begin{equation*}
\|u_2\|_{\sH_0^1(\Omega)}\le C+a\|u_0\|_{\sH_0^1(\Omega)}\le C\,'
\end{equation*}
with $C\,'$ independent of $B$ by estimate (\ref{estimaciona}).

\end{proofof}

\begin{proofof}{Proposition \ref{estimation:reste: 2}}
We compute
\begin{align}\label{calculord2}
&-\dive{\left(\alpha(1+\chi_{B}\varepsilon) \nabla (u_{0}+
\varepsilon u_{1}+\varepsilon^2 u_2) \right)} -(\lambda_{0}
+\varepsilon\lambda_{1}+\varepsilon^2\lambda_2)(u_{0}+
\varepsilon u_{1}+\varepsilon^2 u_2) \nonumber\\
&= -\alpha \Delta u_{0} -\lambda_{0} u_{0}  + \varepsilon \left(
-\alpha\Delta u_{1} - \lambda_{0}u_{1}-\lambda_{1}
u_{0}-\dive{(\alpha\chi_{B} \nabla u_{0})} \right) \nonumber\\
&\ \ \ \ +\varepsilon^2\left( -\alpha\Delta
u_2-\lambda_0u_2-\lambda_1u_1-\lambda_2u_0 -\dive{(\alpha\chi_{B}
\nabla u_{1})}\right) \nonumber\\
&\ \ \ \ +\varepsilon^3 \left(-\lambda_{1} u_{2}-\lambda_2 u_1
-\dive{(\alpha\chi_{B}\nabla u_{2})}\right)+\varepsilon^4(\lambda_2
u_2) \nonumber\\
&=\varepsilon^3 \left(-\lambda_{1} u_{2}-\lambda_2 u_1
-\dive{(\alpha\chi_{B}\nabla u_{2})}\right)+\varepsilon^4(\lambda_2
u_2)
\end{align}
using equations (\ref{ecuacionu0}), and (\ref{eq5}) for $i=1,2$.
Then, since
$$\|-\dive{(\alpha\chi_{B}\nabla
u_{2})}\|_{\sH^{-1}(\Omega)}\le\alpha\,
\|u_2\|_{\sH_0^1(\Omega)},$$
it follows from equation (\ref{calculord2}) and estimates (\ref{estimativolanda1}),  (\ref{estimacionu1}), and (\ref{estimacionu2}),
that for $\varepsilon\ll 1$,
\begin{align}\label{estimacionord2}
\|-&\dive{\left(\alpha(1+\chi_{B}\varepsilon) \nabla (u_{0}+
\varepsilon u_{1}+\varepsilon^2 u_2) \right)} -(\lambda_{0}
+\varepsilon\lambda_{1}+\varepsilon^2\lambda_2)(u_{0}+ \varepsilon
u_{1}+\varepsilon^2 u_2)\|_{\sH^{-1}(\Omega)} \nonumber\\
&\le\Bigl((\alpha+\lambda_1)\|u_2\|_{\sH_0^1(\Omega)}+|\lambda_2|\,
\|u_1\|_{\sH_0^1(\Omega)}\Bigr)\varepsilon^3+(|\lambda_2|\,\|u_2\|_{\sH_0^1(\Omega)})
\varepsilon^4 \nonumber\\
&\le C_1 \varepsilon^3+C_2\varepsilon^4\le C\varepsilon^3,
\end{align}
\noindent Moreover, one has
\begin{align*}
\|u_0+\varepsilon u_1+\varepsilon^2 u_2\|_{\sH^{-1}(\Omega)} &= \sup_{\varphi \in\sH_0^1(\Omega)}\cfrac{\displaystyle\int_{\Omega}(u_0+\varepsilon u_1+\varepsilon^2 u_2)\varphi}{\|\varphi\|_{\sH_0^1(\Omega)}}\ge\cfrac{\displaystyle\int_{\Omega}(u_0+\varepsilon u_1+\varepsilon^2 u_2)u_0}{\|u_0\|_{\sH_0^1(\Omega)}}\\
&= \cfrac{\displaystyle\int_{\Omega}u_0^2+\varepsilon^2\int_{\Omega}u_0u_2}{\|u_0\|_{\sH_0^1(\Omega)}}\,.
\end{align*}
Then, using relation (\ref{prodint}), we obtain
\begin{align*}
\|u_0+\varepsilon u_1+\varepsilon^2 u_2\|_{\sH^{-1}(\Omega)} &\ge \cfrac{1-\cfrac{\varepsilon^2}2 \displaystyle\int_{\Omega}u_1^2}{\|u_0\|_{\sH_0^1(\Omega)}} \ge \frac{1-\frac{\varepsilon^2}2\,C^2}{\|u_0\|_{\sH_0^1(\Omega)}},
\end{align*}
since $u_{1}$ is bounded in $H^1_0(\Omega)$ and consequently, in $L^2(\Omega)$ as shown in  (\ref{estimacionu1}). For $\varepsilon<\frac 1C$, we get
\begin{equation}\label{estimacporarriba}
\|u_0+\varepsilon u_1+\varepsilon^2 u_2\|_{\sH^{-1}(\Omega)}\ge\frac{1}{2\|u_0\|_{\sH_0^1(\Omega)}}=\frac 12 \sqrt{\frac {\alpha}{\lambda_0}}
\end{equation}
By (\ref{estimacionord2}) and (\ref{estimacporarriba}), we then have for $\varepsilon<1/C$ small enough
\begin{align}\label{estimacionfinal}
\|-&\dive{\left(\alpha(1+\chi_{B}\varepsilon) \nabla (u_{0}+
\varepsilon u_{1}+\varepsilon^2 u_2) \right)} -(\lambda_{0}
+\varepsilon\lambda_{1}+\varepsilon^2\lambda_2)(u_{0}+ \varepsilon
u_{1}+\varepsilon^2 u_2)\|_{\sH^{-1}(\Omega)} \nonumber\\
&\le 2\,C\varepsilon^3 \sqrt{\frac {\lambda_0}{\alpha}}\, \|u_0+\varepsilon u_1+\varepsilon^2 u_2\|_{\sH^{-1}(\Omega)}.
\end{align}
By the quasimode argument, there is an element of the spectrum of $-\dive(\alpha(1+\chi_B\varepsilon)\nabla\cdot)$ in 
$H^{-1}(\Omega)$ whose distance from  $\lambda_0+\varepsilon\lambda_1+\varepsilon^2\lambda_2$ is atmost $2\,C\varepsilon^3\sqrt{\frac{\lambda_0}{\alpha}}$. By similar arguments as those at the end of Proposition \ref{estimation:reste:1}, one concludes that such an element is precisely $\lambda_\varepsilon$, the first eigenvalue of $-\dive(\alpha(1+\chi_B\varepsilon)\nabla\cdot)$. 
\end{proofof} 

%%%%%%%%%%%%%%%%%%%%%%%%%%%%%%%%%%%%%%%%%%%%
%%%%% \input{sec2_r}

\section{Minimization of the second order approximation of $\lambda(B)$}\label{relaxation}
\setcounter{equation}{0}
\setcounter{theorem}{0}

Although our main interest is to minimize the ground state $\lambda_{\varepsilon}$ with respect to the set $B$, given $\varepsilon >0$, the general feeling is that the optimization problem is not well posed. A relaxed problem which is not so simple to describe was obtained in Cox and Lipton \cite{CoxLipton}.  In order to understand the nature of the problem for small contrasts Conca et. al. used a first order approximation  \cite{ConcaLaurainMahadevan}. Indeed, after proving a slightly weaker estimate as compared  to Proposition  \ref{estimation:reste:1} using a more ad hoc method of estimation, they conclude that 
\begin{equation} \label{asyminpb: order1}
\displaystyle{\left| \inf_B  \lambda_\varepsilon(B) -
  \lambda_0 -\varepsilon \inf_B  \lambda_1(B) \right| \le C
  \varepsilon^{\frac{3}{2}}\,.} 
\end{equation}
This permits to obtain approximate minimizers for the eigenvalue functional $\lambda_\varepsilon$ by minimizing, instead, the functional $\lambda_0+\varepsilon \lambda_1$. This is a well posed problem and since the original problem may not be well posed it may fail to capture some of the features of the original minimization problem. With this motivation, we go further and do a second order approximation.  Indeed, Proposition \ref{estimation:reste: 2} allows us to conclude that 
\begin{equation} \label{asyminpb: order2}
\displaystyle{\left| \inf_B  \lambda_\varepsilon(B) - \inf_B (\lambda_0 + \varepsilon \lambda_1(B) + \varepsilon^2 \lambda_2(B)) \right| \le C\varepsilon^3\,.}
\end{equation}
Thus, we can obtain approximate minimizers for the functional $\lambda_\varepsilon$, for given $\varepsilon >0$ small enough, by minimizing the functional $\lambda_0 + \varepsilon \lambda_1 + \varepsilon^2 \lambda_2$ which is a second order approximation of $\lambda_\varepsilon$. We then study the problem:
$$\text{minimize}\ \{\lambda_0+\varepsilon\lambda_1(B)+\varepsilon^2\lambda_2(B)\ ;\ B\subseteq\Omega,\,|B|=m\},\quad 0<m<|\Omega|,\ m\ \text{fixed}$$
or equivalently
$$\text{minimize}\ \{\lambda_1(B)+\varepsilon\lambda_2(B)\ ;\ B\subseteq\Omega,\,|B|=m\},$$
since $\lambda_0$ is independent of $B$ and $\varepsilon>0$ is fixed.
From the expressions for $\lambda_1(B),\lambda_2(B)$ computed in the previous section, we finally  consider the problem
$$\text{minimize }F(\chi):=\alpha \int_{\Omega}\chi(\nabla u_0+\varepsilon\nabla v(\chi))\cdot\nabla u_0
$$
over the class of admissible domains represented by their characteristic functions
$$  \mathcal U_{ad}:=\{\chi\ ;\ \chi=\chi_B,\, B\subseteq\Omega,\,|B|=m\}\subseteq\sL^{\infty}(\Omega),
$$
and $v=v(\chi)\in\sH_0^1(\Omega)$ satisfisfies
\begin{eqnarray}
  -\alpha\Delta v-\lambda_0 v=\lambda_1(\chi) u_0+\dive(\alpha\chi\nabla u_0),\\
  \lambda_1(\chi):=\int_{\Omega}\alpha\chi |\nabla u_0|^2,\\
  v\perp u_0\ {\text{in}}\ \sL^{2}(\Omega).\nonumber
\end{eqnarray}

\subsection{Relaxation of the minimization problem}

The functional $F$ is lower-semicontinuous for the weak-$*$ topology on $L^{\infty}(\Omega)$, being quadratic with respect to $\chi$, but the admissible set $\mathcal U_{ad}$ is not closed for this topology. In order to have a well-posed minimization problem we need to work on the closure $\overline{{\mathcal U}_{ad}}$ and calculate the {\em lower semicontinuous envelope} of $F$ with respect to the weak-$*$ topology on $L^{\infty}(\Omega)$.
\begin{equation*}
\bar{F}(\theta):=\inf\{\liminf F(\chi_n)\ :\ \chi_n\rightharpoonup\theta \ \text{in}\ \sL^{\infty}(\Omega)^{\ast}\},\ \theta\in{\overline{{\mathcal U}_{ad}}},
\end{equation*}
where
$$\overline{{\mathcal U}_{ad}}={\overline{\mathcal U_{ad}}}^{\,L^{\infty}(\Omega)^{\ast}}=\{\theta\in L^{\infty}(\Omega)\ ;\ 0\le\theta\le1,\,\int_{\Omega}\theta=m\}.$$
We shall follow the general procedure to compute $\bar{F}$ and obtain the following theorem. 
\begin{theorem}
\label{fonctionnelle:relaxee}
For any $\theta \in \overline{{\mathcal U}_{ad}}$, we have
$$\bar{F}(\theta)=\alpha \int_{\Omega} \theta \left[\nabla u_{0}+\varepsilon \nabla v_{\infty}(\theta)\right].\nabla u_{0} -\varepsilon \theta(1-\theta) |\nabla u_{0}|^2,$$
where $v_{\infty}(\theta) \in\sH^1_{0}(\Omega)$ is solution of
\begin{eqnarray}
  -\alpha\Delta v-\lambda_0 v=\lambda_1(\theta) u_0+\dive(\alpha\theta \nabla u_0),\\
  \lambda_1(\theta):=\int_{\Omega}\alpha\theta |\nabla u_0|^2,\\
  v\perp u_0\ {\text{in}}\ \sL^{2}(\Omega).\nonumber
\end{eqnarray}
\end{theorem}
The proof of the Theorem \ref{fonctionnelle:relaxee} will use some results on $H$-measures. This tool was introduced by P. G\'erard in \cite{Gerard} and L. Tartar in \cite{Tartar} to understand the obstruction to compactness via  a matrix of complex-valued Radon measures $(\mu_{ij}(x,\xi))_{1\leq i,j\leq p}$ on $\R^N\times\mathbb{S}^{N-1}$ on the space-frequency domain associated to weakly convergent sequences.  We refer to the two previous references for a complete presentation of $H$-measures and to \cite{AllaireGutierrez} for their applications in small contrast homogenization. We will need the two following results (Theorem 2-2 and Lemma 2-3 in \cite{AllaireGutierrez}).
\begin{theorem}{\bf \cite{AllaireGutierrez}}
\label{Hmesure:theoreme1}
Let $u_{\varepsilon}$ be a sequence which weakly converges to $0$ in $\sL^2(\R^N)^p$. There exists a subsequence and a $H$-measure $\mu$ such that
$$\lim_{\varepsilon\rightarrow 0} \int_{\R^N} q(u_{\varepsilon}).\bar{u}_{\varepsilon} =\int_{\R^N} \int_{\mathbb S^{N-1}} \sum_{i,j=1}^p q_{ij}(x,\xi) \mu_{ij}(dx,d\xi)$$
for any polyhomogeneous pseudo-differential operator $q$ of degree $0$ with symbol $(q_{ij}(x,\xi))$.
\end{theorem}
We shall also use the following lemma due to Kohn and Tartar that deals with the special case of sequences of characteristic functions.
\begin{lemma}{\bf \cite{AllaireGutierrez}}
\label{Hmesure:theoreme2}
Let $\chi_{\varepsilon}$ be a sequence of characteristic functions that weakly-$*$ converges to some $\theta$ in $\sL^\infty(\Omega,[0,1])$. Then the corresponding $H$-measure $\mu$ for the sequence $(\chi_{\varepsilon}-\theta)$ is necessarily of the type
$$\mu(dx,d\xi)= \theta(x) (1-\theta(x)) \nu(dx,d\xi),$$
where, for a given $x$, the measure $\nu(dx,d\xi)$ is a probability measure with respect to $\xi$.

Conversely, for any such probability measure $\nu\in \mathcal{P}(\Omega,\mathbb S^{N-1})$, there exists a sequence $\chi_{\varepsilon}$ of characteristic functions which weakly-$*$ converges to $\theta \in \sL^\infty(\Omega,[0,1])$ such that $\theta(1-\theta)\nu$ is the $H$-measure of $(\chi_{\varepsilon}-\theta)$.
\end{lemma}
\begin{proofof}{Proposition \ref{fonctionnelle:relaxee}} Let $\theta \in {\overline{{\mathcal U}_{ad}}}$.  Let $\{\chi_n\}$ be a sequence in $\mathcal U_{ad}$ such that
\begin{equation}\label{cvchin}
 \chi_n\overset{\star}{\rightharpoonup}\theta\in{\overline{{\mathcal U}_{ad}}}.
\end{equation}
We then analyze the limit of
\[
 F(\chi_n)=\alpha \underbrace{\int_{\Omega}\chi_n|\nabla u_0|^2}_{A_n}+\alpha \, \varepsilon  \underbrace{\int_{\Omega}\chi_n\nabla v_n\cdot\nabla u_0}_{B_n},
\]
with $v_n:=v(\chi_n)\in\sH_0^1(\Omega)$ such that
\begin{eqnarray}
  -\alpha\Delta v_n-\lambda_0 v_n=\lambda_1(\chi_n)u_0+\dive(\alpha\chi_n\nabla u_0),\label{eqvn}\\
  \lambda_1(\chi_n)=\int_{\Omega}\alpha\chi_n|\nabla u_0|^2,\nonumber\\
  v_n\perp u_0\ \text{in}\ \sL^2(\Omega).\label{orto}
\end{eqnarray}
{\sc Step 1: } Passing to the limit in  $A_{n}$ is easy.  By the convergence (\ref{cvchin}), we have
\begin{equation}\label{lim1}
A_n=\lambda_1(\chi_n)\longrightarrow \alpha \int_{\Omega}\theta\,|\nabla u_0|^2=\lambda_1(\theta).
\end{equation}

\noindent 
{\sc Step 2: } Now we study the limit of the sequence $v_{n}$. By (\ref{orto}), we know that
\begin{equation*}
\left(1-\frac{\lambda_0}{\lambda^1}\right)\int_{\Omega}|\nabla v_n|^2\le C,
\end{equation*}
using a similar estimation as (\ref{estimacion}). Then $\|v_n\|_{\sH_0^1}^2\le C $ and hence, 
\begin{equation*}
v_n\rightharpoonup v_{\infty}=v_{\infty}(\theta)\quad\text{weak-}\sH_0^1(\Omega)
\end{equation*}
up to a subsequence. Since $\sH_0^1(\Omega)$ is compactly embedded in $\sL^2(\Omega)$,
\begin{equation*}
v_n\longrightarrow v_{\infty}\quad{\text{in}}\ \sL^2(\Omega)
\end{equation*}
up to a subsequence. Therefore, we can pass to variational limit from (\ref{eqvn}) to obtain,
\begin{equation}\label{eqvinf}
-\alpha\Delta v_{\infty}-\lambda_0 v_{\infty}=\lambda_1(\theta)u_0+\dive(\alpha\,\theta\nabla u_0).
\end{equation}
Moreover, passing to the limit from (\ref{orto}), we have
\[
  v_{\infty}\perp u_0\ \text{in}\ \sL^2(\Omega),
\]
accordingly, since $\|u_0\|_{\sL^2}=1$, $v_{\infty}=v_{\infty}(\theta)$ is uniquely defined in (\ref{eqvinf}) and $v_{\infty}$ depends (linearly) only on $\theta$ and not on the convergent subsequence of $\{v_n\}$. \par

\noindent
{\sc Step 3: }
\noindent The main difficulty is to pass to the limit in $B_{n}$ which is quadratic with respect to $\chi_{n}$. 
First, we can rewrite $B_n$ as
\begin{equation}\label{eqbn}
B_n=\int_{\Omega}\chi_n\nabla w_n\cdot\nabla u_0+\int_{\Omega}\chi_n\nabla z_n\cdot\nabla u_0,
\end{equation}
$w_n,z_n\in\sH_0^1(\Omega)$ such that
\begin{equation}\label{eqwn}
-\alpha\Delta w_n=\lambda_0 v_n+\lambda_1(\chi_n)u_0,
\end{equation}
\begin{equation}\label{eqzn}
-\Delta z_n=\dive(\chi_n\nabla u_0).
\end{equation}
On the one hand, since
\begin{equation*}
\lambda_0 v_n+\lambda_1(\chi_n)u_0\longrightarrow\lambda_0 v_{\infty}+\lambda_1(\theta)u_0\quad \text{in}\ \sL^2(\Omega),
\end{equation*}
(\ref{eqwn}) implies
\begin{equation*}
w_n\longrightarrow w\quad\text{in}\ \sH_0^1(\Omega),
\end{equation*}
where $w\in\sH_0^1(\Omega)$ satisfies the equation
\begin{equation*}
-\alpha\Delta w=\lambda_0 v_{\infty}+\lambda_1(\theta)u_0
\end{equation*}
and, in consequence,
\begin{equation}\label{lim21}
\int_{\Omega}\chi_n\nabla w_n\cdot\nabla u_0\longrightarrow\int_{\Omega}\theta\,\nabla w\cdot\nabla u_0.
\end{equation}
The difficulty is now  to calculate the limit in the second term of $B_n$ in (\ref{eqbn}). We observe that $\dive\chi_n\nabla u_0\rightharpoonup \dive\theta\nabla u_0$ weakly in $\sH^{-1}(\Omega)$ and since $(-\Delta)^{-1}$ is a isomorphism from $\sH^{-1}(\Omega)$ into $\sH_0^1(\Omega)$, we get  $\sL^2$-weak convergence of $\nabla z_n$. However, this is not enough for passing to the limit in the second term of $B_n$ because, in the  product $\chi_n\nabla z_n$, both sequences $\chi_n$ and $\nabla z_n$ only converge weakly. For handling this convergence problem we use the results on $H$-convergence stated before. \par

\noindent
{\sc Step 4: }For simplicity if $\Omega$ is $\mathbb{R}^n$, in view of Theorem \ref{Hmesure:theoreme1} and Lemma \ref{Hmesure:theoreme2}, the limit of the second term in (\ref{eqbn})  becomes
\begin{equation*}
\lim_{n\to\infty}\int_{\R^N}\chi_n\nabla z_n\cdot\nabla
u_0=\int_{\R^N}\theta\, Q(\theta)\cdot\nabla
u_0-\int_{\R^N}\theta(1-\theta)M\nabla u_0\cdot\nabla u_0,
\end{equation*}
where the pseudo-differential operator $Q$ is defined in Lemma \ref{pseudo:ordre0} (in the appendix) and it's symbol has been calculated therein and, 
\[M=\int_{\mathbb{S}^{N-1}}\xi\otimes\xi\,\nu(x,\dif\xi),\]
$\nu=\nu(x,\xi)$ is a probability measure with respect to $\xi$ that depends on the sequence $\{\chi_n\}$ and $Q(\theta)=\nabla z$ with $z \in \sH_0^1(\Omega)$ verifies the equation
\begin{equation*}
-\Delta z=\dive (\theta\,\nabla u_0).
\end{equation*}
\noindent
{\sc Step 5: }
But we need to work on $\Omega$ bounded. To that end, we use a localization procedure. This argument proceeds as follows. Let $(\zeta_k)$ be a sequence of smooth compactly supported functions in $C_0^\infty(\R^N)$ such that supp$\,\zeta_k\subset\Omega$ for all $k$ and $\zeta_k$ converges to 1 strongly in $\sL^2(\Omega)$. Then the second term on the right hand side of (\ref{eqbn}) can be written as
\begin{equation}\label{rewritwithzk}
  \int_{\Omega}\chi_n\nabla z_n\cdot\nabla u_0=\int_{\R^N}\zeta_k\chi_n\nabla z_n\cdot\nabla u_0+\int_{\R^N}(1-\zeta_k)\chi_n\nabla z_n\cdot\nabla u_0.
\end{equation}

Note that the last term in (\ref{rewritwithzk}) converges to 0 uniformly with respect to $n$ when $k$ tends to infinity because $z_n$ is bounded in $\sH^1(\Omega)$. We now fix $k$ and consider another smooth compactly supported function $\psi_k\in C_0^{\infty}$ such that $\psi_k\equiv 1$ inside the support of $\zeta_k$. The first term on the right hand side of (\ref{rewritwithzk}) is thus equal to
\begin{equation}\label{lastintegral}
\int_\Omega\zeta_k(\psi_k\chi_n)\nabla(\psi_k z_n)\cdot\nabla u_0.
\end{equation}

Rewriting the equation (\ref{eqzn}) in $\R^N$ as
\[-\Delta (\psi_k z_n)-\Delta((1-\psi_k)z_n)=\dive(\psi_k\chi_n\nabla u_0)+\dive((1-\psi_k)\chi_n\nabla u_0),\]
we can show that the function $\psi_k z_n$ is the sum of $\tilde z_n,\check z_n$ on the support of $\zeta_k$ being $\tilde z_n,\check z_n$ solutions of the following equations in the whole space $\R^N$
\begin{align*}
  -\Delta\tilde z_n &= \dive\psi_k\chi_n\nabla u_0\quad\text{in}\ \R^N,\\
  \Delta\check z_n &= \dive z_n\nabla\psi_k+\nabla\psi_k\cdot(\chi_n\nabla u_0+\nabla z_n)\quad\text{in}\ \R^N.
\end{align*}
We then notice that
$$\dive\psi_k\chi_n\nabla u_0\rightharpoonup\dive\psi_k\theta\nabla u_0\quad \text{ weakly in } \sH^{-1}(\R^N)\ \text{and}$$
$$\dive z_n\nabla\psi_k+\nabla\psi_k\cdot(\chi_n\nabla u_0+\nabla z_n)\rightarrow\dive z\nabla\psi_k+\nabla\psi_k\cdot(\theta\nabla u_0+\nabla z)\quad\text{ strongly in } \sH^{-1}(\R^N)$$
since this last term clearly converges weak-$\sL^2(\Omega)$. Using the fact that $(-\Delta)^{-1}$ is an isomorphism from $\sH^{-1}(\R^N)$ into $H^1(\R^N)$, we thus have
\begin{equation*}
  \tilde z_n\rightharpoonup\tilde z\quad\text{weakly in }H^1(\R^N)
\end{equation*}
and
\begin{equation*}
  \check z_n\rightarrow\check z\quad\text{strongly in }H^1(\R^N)
\end{equation*}
where $\tilde z,\check z$ verify
\begin{align*}
  -\Delta\tilde z &= \dive\psi_k\theta\nabla u_0\quad\text{in}\ \R^N,\\
  \Delta\check z &= \dive z\nabla\psi_k+\nabla\psi_k\cdot(\theta\nabla u_0+\nabla z)\quad\text{in}\ \R^N.
\end{align*}
Obviously $z=\tilde z+\check z$ on the support of $\zeta_k$.\par
Now noting that the integral (\ref{lastintegral}) has close relationship with the formulation of the $H$-measures, we see that, as in the whole space case, $\nabla\tilde z_n$ depends linearly on $(\psi_k\chi_n)$ through the pseudo-differential operator $Q$ of symbol (\ref{symbol}). Therefore applying Theorem 2 of \cite{Gerard}, we conclude that the limit of the first term on the right hand side of (\ref{rewritwithzk}) is equal to
\begin{multline*}
  \lim_{n\to\infty}\int_{\R^N}\zeta_k(\psi_k\chi_n)\nabla(\check z_n+\tilde z_n)\cdot\nabla u_0 = \int_{\R^N}\zeta_k(\psi_k\theta)\nabla\check z\cdot\nabla u_0\\
  \quad+\lim_{n\to\infty}\int_{\R^N}\zeta_k(\psi_k\chi_n)\nabla\tilde z_n\cdot\nabla u_0\\
  =\int_{\R^N}\zeta_k(\psi_k\theta)\nabla\check z\cdot\nabla u_0+\int_{\R^N}\zeta_k(\psi_k\theta)\nabla\tilde z\cdot\nabla u_0-\int_{\R^N}\zeta_k\psi_k\theta(1-\theta)M\nabla u_0\cdot\nabla u_0\\
  =\int_{\Omega}\zeta_k\theta\nabla z\cdot\nabla u_0-\int_{\Omega}\zeta_k\theta(1-\theta)M\nabla u_0\cdot\nabla u_0.
\end{multline*}
Finally making $k$ tends to $\infty$, we obtain the desired bounded domain case.\par
We go back to the calculation of the limit in (\ref{eqbn}). Indeed, gathering the limit (\ref{lim21}) and limit calculated above, it follows that
\begin{equation}\label{lim2}
  \lim_{n\to\infty}B_n=\int_{\Omega}\theta\,\nabla v_{\infty}(\theta)\cdot\nabla u_0-\int_{\Omega}\theta(1-\theta)\int_{\mathbb{S}^{N-1}}(\xi\cdot\nabla
  u_0)^2\,\nu(\dif x,\dif\xi).
\end{equation}
From (\ref{lim1}) and (\ref{lim2}), finally one has
\begin{align*}
  \lim_{n\to\infty} F(\chi_n) &= \lim_{n\to\infty}A_n+\varepsilon\lim_{n\to\infty}B_n\\
  &= \frac 1{\alpha}\,\lambda_1(\theta)+\varepsilon\int_{\Omega}\theta\,\nabla v_{\infty}(\theta)\cdot\nabla u_0
  -\varepsilon\int_{\Omega}\theta(1-\theta)\int_{\mathbb{S}^{N-1}}(\xi\cdot\nabla u_0)^2\,\nu(\dif x,\dif\xi).
\end{align*}
\noindent
{\sc Step 6: }Now we calculate
\[
  \bar{F}(\theta)=\inf_{\nu} \lim F(\chi_n).
\]
To that end, we notice that
\begin{equation*}
\int_{\mathbb{S}^{N-1}}(\xi\cdot\nabla u_0)^2\,\nu(\dif x,\dif\xi)\le|\nabla u_0|^2(x)\quad a.e.\ x\in\Omega,
\end{equation*}
since $\nu$ is a probability measure with respect to $\xi\ a.e.\ x\in\Omega$. Moreover, this value is reached when we take the Dirac measure $\delta_{\nabla u_0(x)}$, i.e., when
\[\nu(x,\xi)=\delta_{\xi_x}\dif x,\quad \xi_x=\nabla u_0(x).\]
From the converse part of Lemma 2.3 in \cite{AllaireGutierrez}, the minimum for
\[\inf_{\nu} \lim F(\chi_n)\]
is also achieved. So, finally we can conclude
\begin{equation}\label{calfin}
\bar{F}(\theta)=\int_{\Omega}\theta\,(\nabla u_0+\varepsilon\nabla v_{\infty}(\theta))\cdot\nabla u_0-\varepsilon\int_{\Omega} \theta\,(1-\theta)\,|\nabla u_0|^2.
\end{equation}
Recall that $v_{\infty}=v_{\infty}(\theta)$ depends linearly on $\theta$.
\end{proofof}

\subsection{Optimality conditions for the relaxed problem.}

The relaxed functional $\bar F$ achieves it's minimum of $\overline{\mathcal U_{ad}}$ since it is lower-semicontinuous and the constraint set is compact for the weak-$*$ topology. We first investigate the differentiability properties of $\bar F$ in order to obtain optimality conditions for a minimizer of $\bar F$ on the compact convex set $\overline{\mathcal U_{ad}}$. 
\begin{proposition}\label{calcul:derivees:formulation:relaxee} The functional $\bar F$ is Fr\'echet differentiable of every order and we have the following expressions for the Gateaux derivatives of first and second order
\begin{equation}\label{calcul:gradient}
\bar F'(\theta)\varphi=\int_{\Omega}\Bigl[2\varepsilon(\nabla v_{\infty}(\theta)+\theta\nabla u_0)+(1-\varepsilon)\nabla u_0\Bigr]\cdot\nabla u_0\,\varphi.
\end{equation}
and
\begin{equation} \label{calcul:derivee:seconde}
\bar{F}''(\theta)(\varphi,\varphi)=2\varepsilon\int_{\Omega}(\nabla v_{\infty}(\varphi)+\varphi\nabla u_0)\cdot\nabla u_0\,\varphi.
\end{equation}
\end{proposition}
\begin{proof} The linearity of the application $\theta \mapsto v_\infty(\theta)$ and the expression for $\bar F$ show clearly that it is quadratic with respect to $\theta$. So, the Fr\'echet derivatives exist. In order to calculate the first order derivative, we rewrite (\ref{calfin}) as
$$\bar F(\theta)=\varepsilon\int_{\Omega}\theta \nabla v_{\infty}(\theta)\cdot\nabla u_0+\varepsilon\int_{\Omega}\theta^2|\nabla u_0|^2+(1-\varepsilon)\int_\Omega\theta|\nabla u_0|^2.$$
But, using $v_{\infty}(\theta)$ as test function in (\ref{eqvinf}), we get
\begin{equation*}
 \bar F(\theta)= -\varepsilon\int_\Omega|\nabla v_\infty(\theta)|^2-\frac{\lambda_0}{\alpha}v_{\infty}^2(\theta)+\varepsilon\int_{\Omega}\theta^2|\nabla u_0|^2+(1-\varepsilon)\int_\Omega\theta|\nabla u_0|^2.
\end{equation*}
A simple calculation gives us
\begin{equation*}
  \bar F'(\theta)\varphi=-2\varepsilon\int_\Omega\nabla v_\infty(\theta)\cdot\nabla v_\infty(\varphi)-\frac{\lambda_0}{\alpha}v_{\infty}(\theta)v_{\infty}(\varphi)+2\varepsilon\int_{\Omega} \theta|\nabla u_0|^2\varphi+(1-\varepsilon)\int_\Omega|\nabla u_0|^2\varphi.
\end{equation*}
We now notice that $v_\infty(\varphi)$ satisfies (\ref{eqvinf}). Then, again taking $v_\infty(\theta)$ as test function, we can explicitly write the above expression in terms of $\varphi$ to obtain \eqref{calcul:gradient} then \eqref{calcul:derivee:seconde}.
\end{proof}
We wish to investigate the critical points for the constrained minimization problem  of minimizing $\bar F$ over $\overline{\mathcal U_{ad}}$. To that end, we use the Lagrange's multipliers method with the constraint
$$C(\theta):=\int_\Omega\theta=m,\quad\theta\in\overline{\mathcal U_{ad}} \text{  hence  }
C'(\theta)\varphi=\int_\Omega\varphi.$$
Therefore, the critical points satisfy the Euler-Lagrange equation: for all admissible $\varphi$
\begin{equation*}
  [\bar F'(\theta)+\Lambda C'(\theta)]\varphi=0
\end{equation*}
for some $\Lambda\in\R$; i.e.
\begin{equation*}
  \int_{\Omega}\Bigl[2\varepsilon(\nabla v_{\infty}(\theta)+\theta\nabla u_0)+(1-\varepsilon)\nabla u_0\Bigr]\cdot\nabla u_0\,\varphi+\Lambda\int_\Omega\varphi=0\quad\forall\varphi.
\end{equation*}
Consequently the density of $\overline{\mathcal U_{ad}}$ in $\sL^2(\Omega)$ implies
\begin{equation*}
  2\varepsilon\nabla v_{\infty}(\theta)\cdot\nabla u_0+(2\varepsilon\theta+1-\varepsilon)|\nabla u_0|^2=\Lambda\quad\text{on}\ \Omega.
\end{equation*}
\begin{proposition}
If $\theta^*$ is  optimal in the relaxed formulation, then there is real $\Lambda$ such that:
\begin{equation*}
  2\varepsilon\nabla v_{\infty}(\theta)\cdot\nabla u_0+(2\varepsilon\theta+1-\varepsilon)|\nabla u_0|^2=\Lambda\quad\text{in}\ \Omega.
\end{equation*}
\end{proposition}
Integrating over $\Omega$ and considering $u_0$ as test function in (\ref{eqvinf}), we get the following consequence
\begin{equation*}
  \int_\Omega\nabla v_{\infty}(\theta)\cdot\nabla u_0=0 \text{ and }
  2\varepsilon\int_\Omega\theta|\nabla u_0|^2+\frac{1-\varepsilon}{\alpha}\,\lambda_0=\Lambda|\Omega|.
\end{equation*}

%%%%%%%%%%%%%%%%%%%%%%%%%%%%%%%%%%%%%%%%%%
%%%%% \input{sec3_r}
\section{Numerical illustrations}
\setcounter{equation}{0}
\setcounter{theorem}{0}

In this section we shall illustrate the behavior of the solution of the appoximated problem through numerical simulations. To that end, we place ourselves under assumption of low contrast regime, i.e. $\beta=\alpha(1+\varepsilon)$ for small $\varepsilon$. In the following examples, we will consider $\varepsilon=0.1$ and $\varepsilon=10^{-6}$.

We use an optimization algorithm to minimize $\bar F$: we have implemented a gradient-based steepest descend numerical algorithm for the local proportion $\theta$. At each step of the optimization algorithm, we update the local proportion with a step $\rho_i>0$ by
$$ \theta_i=\min(1,\max(0,\tilde{\theta}_i))\ \text{with}\ \tilde{\theta}_i=\theta_{i-1}-\rho_i\bar F'(\theta_{i-1})+\Lambda_i$$
 where $\Lambda_i$ is the Lagrange multiplier for the volume constraint. The Lagrange multipliers $\Lambda_i$ are approximated at each iteration by simple dichotomy in order to get the constraint $\int_\Omega\theta_i =m$ corresponding to a fixed proportion.

The optimization procedure is coupled with finite elements approximations of the boundary values problems needed to compute both $\bar F$ and its derivative $\bar F'$. To calculate the eigen-pair $(\lambda_0,u_0)$ and all the states $v_{i,\infty}$, we use $\mathbb{P}_2$ finite elements while the local proportions $\theta_i$ have been discretized with $\mathbb{P}_1$.

We will present examples in dimension two and three. The computations have been made with the FEM library FreeFem++ \cite{MR3043640}. The subsequent figures show the local proportion of the material with higher conductivity. We do a comparative analysis in dimension two and three for square and cube cases respectively confirming the mentioned properties in \cite{ConcaLaurainMahadevan} with respect to the distribution of the material with higher conductivity that depends on the shape of the domain $\Omega$.  The volume always refers to the percentage of volume occupied by the higher conductivity material.

\subsection{The square and the cube.}

\indent The computations are made on the unit square $[0,1]^2$ with a regular mesh of $80 000$ triangles. For a very small value of $\varepsilon$ here $10^{-6}$, we have obtained the optimal designs displayed into Figure \ref{square=e-6} for different volume proportions. The dark red region corresponds to $B$ and material $\beta$, the local proportion is then $1$. The blue region correspond to material $\alpha$, the local proportion is then $0$.

\begin{figure}[!htH]

\centering
\begin{subfigure}[b]{.2\textwidth}
 \includegraphics[width=\linewidth]{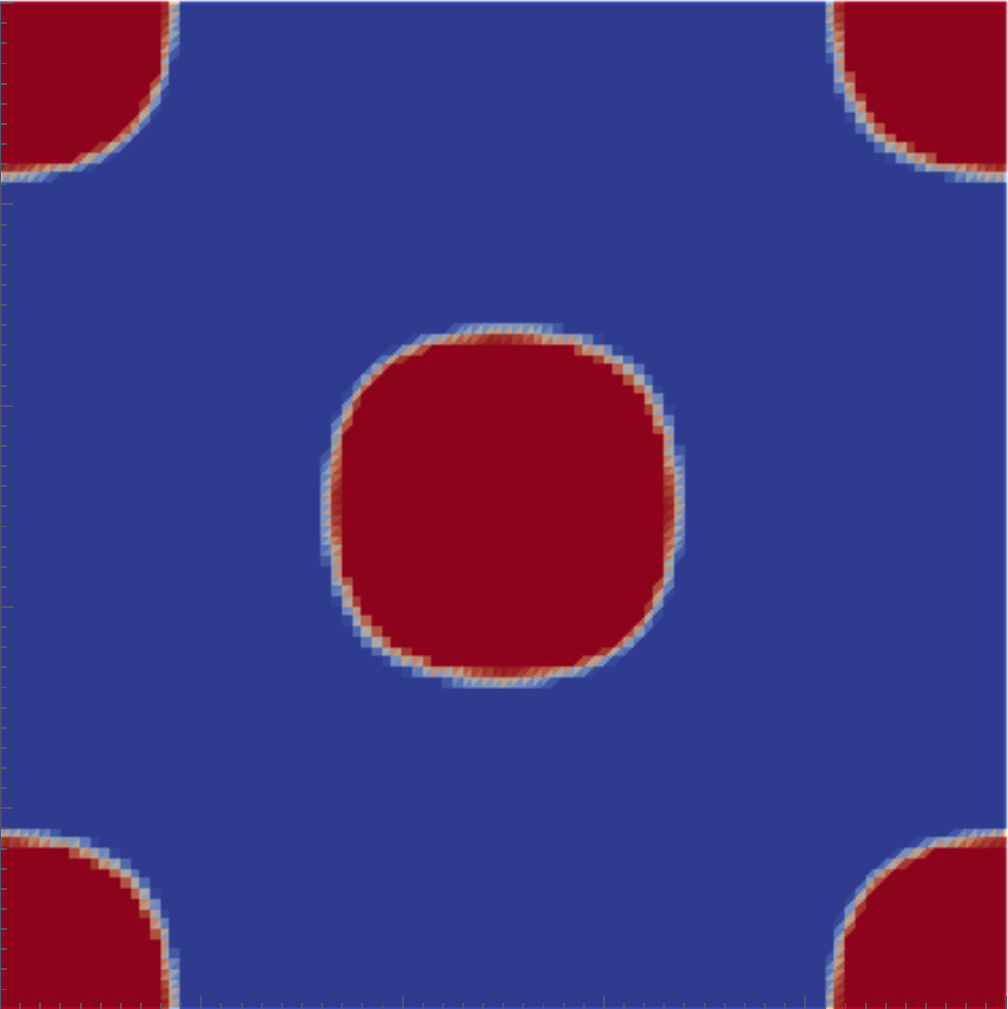}
 \caption{$m/|\Omega|=0.2$}
\end{subfigure}
\begin{subfigure}[b]{.2\textwidth}
 \includegraphics[width=\linewidth]{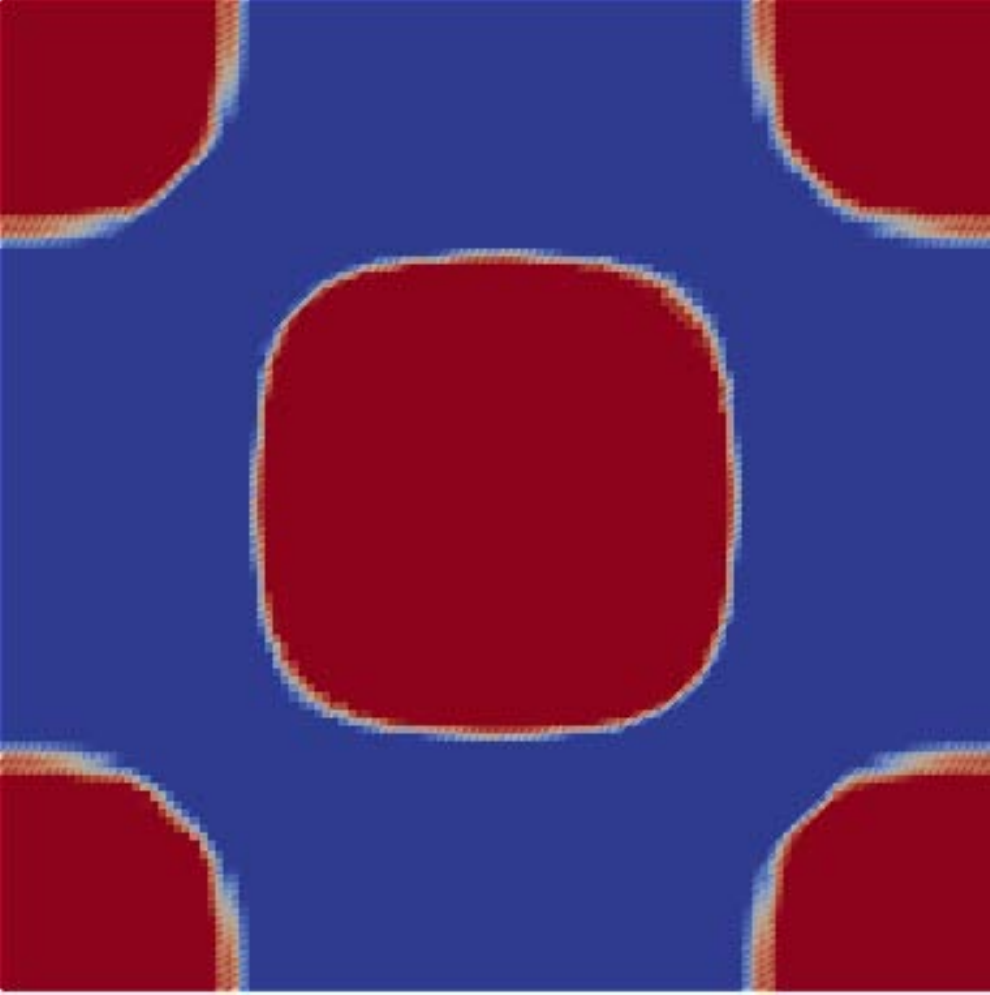}
 \caption{$m/|\Omega|=0.4$}
\end{subfigure}
\begin{subfigure}[b]{.2\textwidth}
 \includegraphics[width=\linewidth]{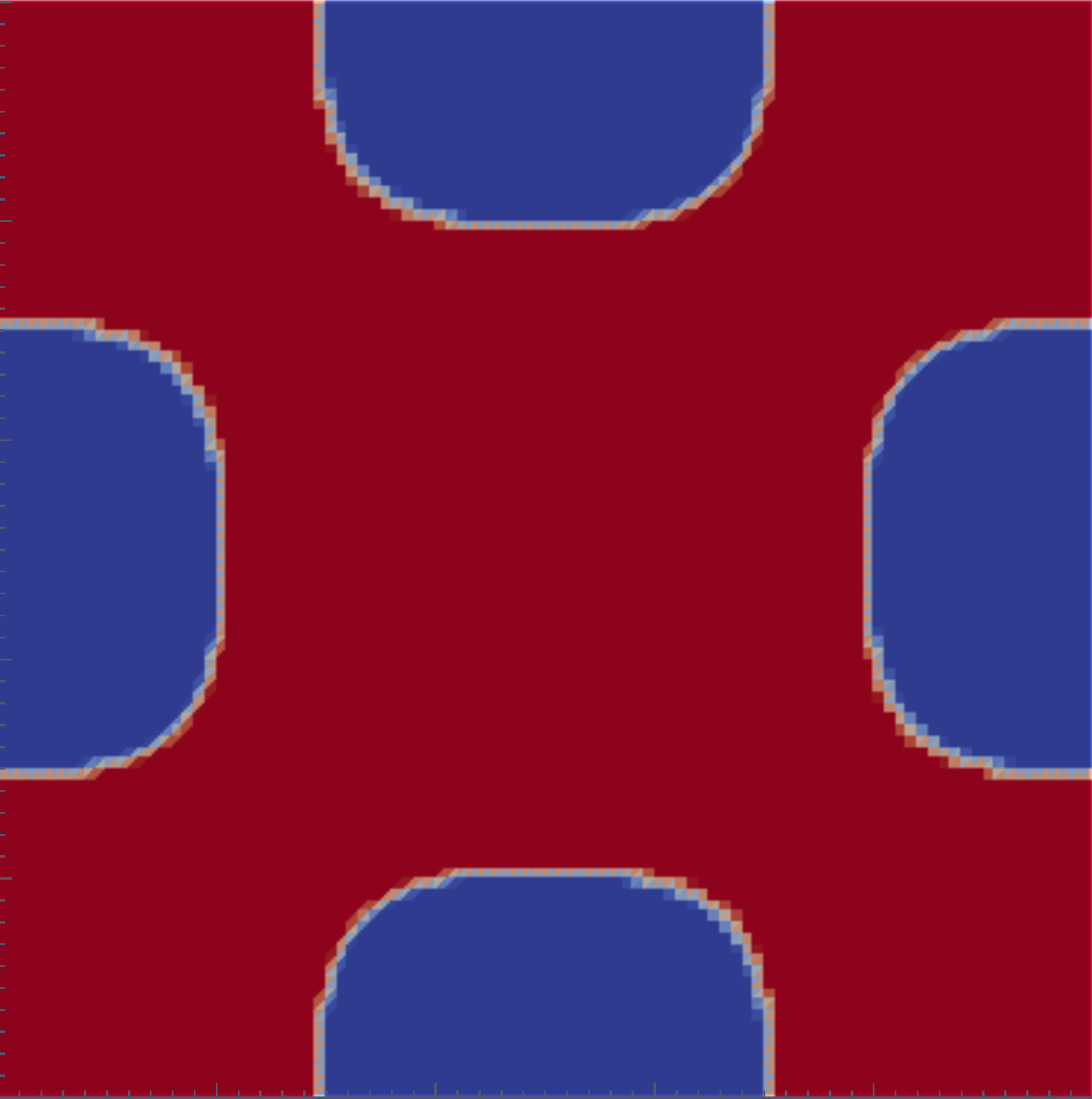}
 \caption{$m/|\Omega|=0.7$}
\end{subfigure}
\begin{subfigure}[b]{.2\textwidth}
 \includegraphics[width=\linewidth]{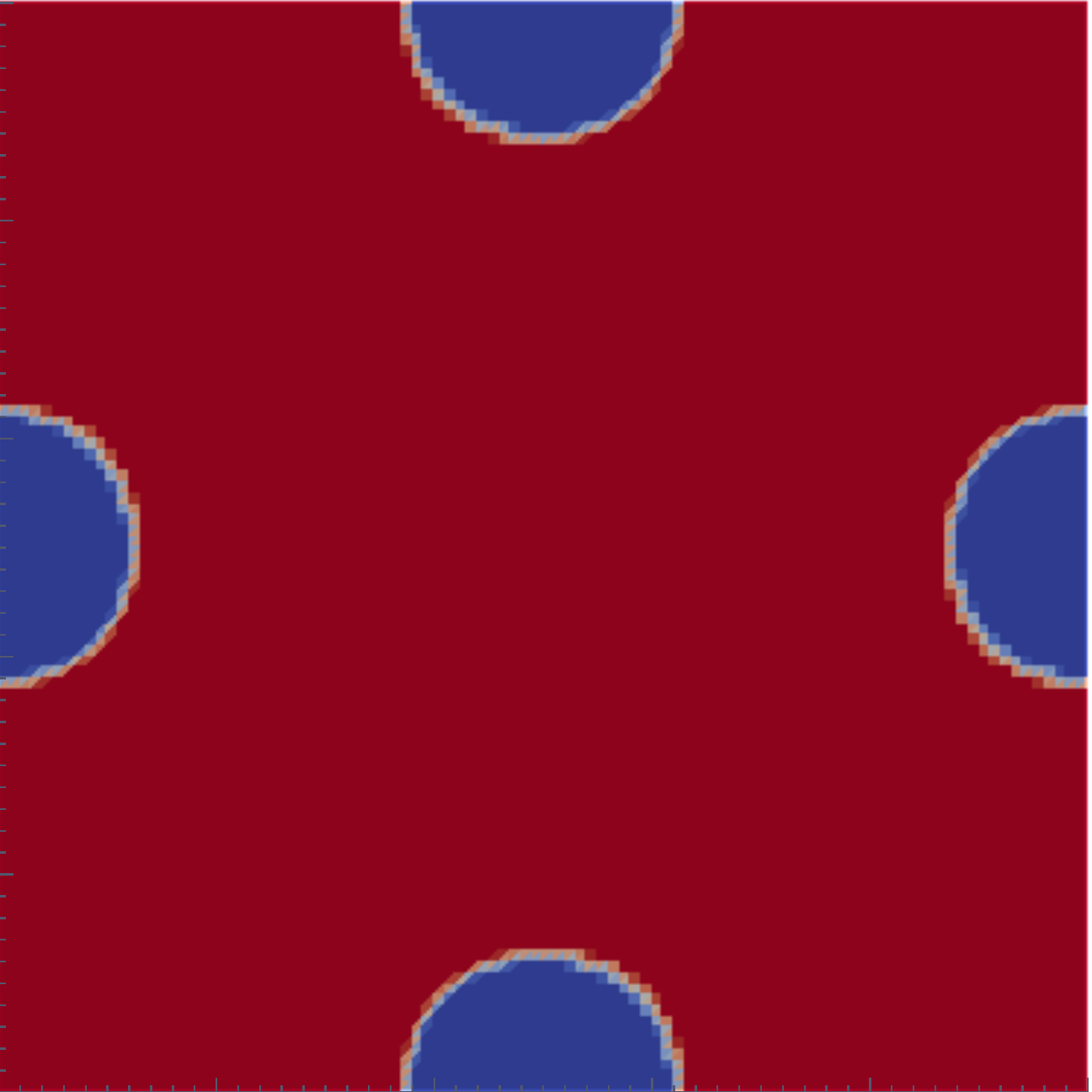}
 \caption{$m/|\Omega|=0.9$}
\end{subfigure}
\begin{subfigure}[b]{0.07\textwidth}
 \vfil
 \includegraphics[width=\linewidth]{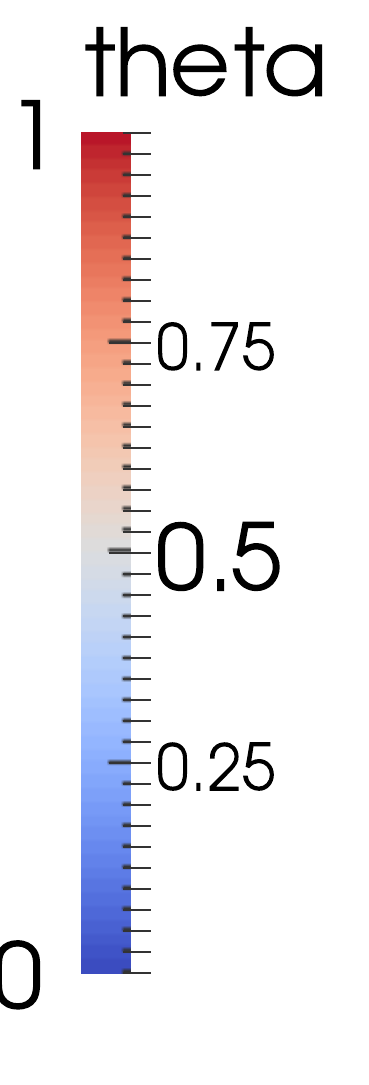}
 \\~
\end{subfigure}
\caption{Nearly optimal distribution $B$ in the square case for $\varepsilon=10^{-6}$ .}
\label{square=e-6}
\end{figure}

The numerically computed optimal region $B$ contains neighborhoods around corners and the center always is also included. Similar results were obtained by Conca, Laurain and Mahadevan in  \cite{ConcaLaurainMahadevan} with a first order approximation only. Nevertheless, the local proportion is very often either $0$ either $1$. Let us now consider the same cases with a much larger parameter $\varepsilon$. In Figure \ref{square=e-1}, we present the results obtained with $\varepsilon=0.1$. We observe that the mixture is much more important: there seems to be a pure material nowhere. We believe that the asymptotic is not reached for such a large value of $\varepsilon$.

\begin{figure}[!htH]

\centering
\begin{subfigure}[b]{.2\textwidth}
 \includegraphics[width=\linewidth]{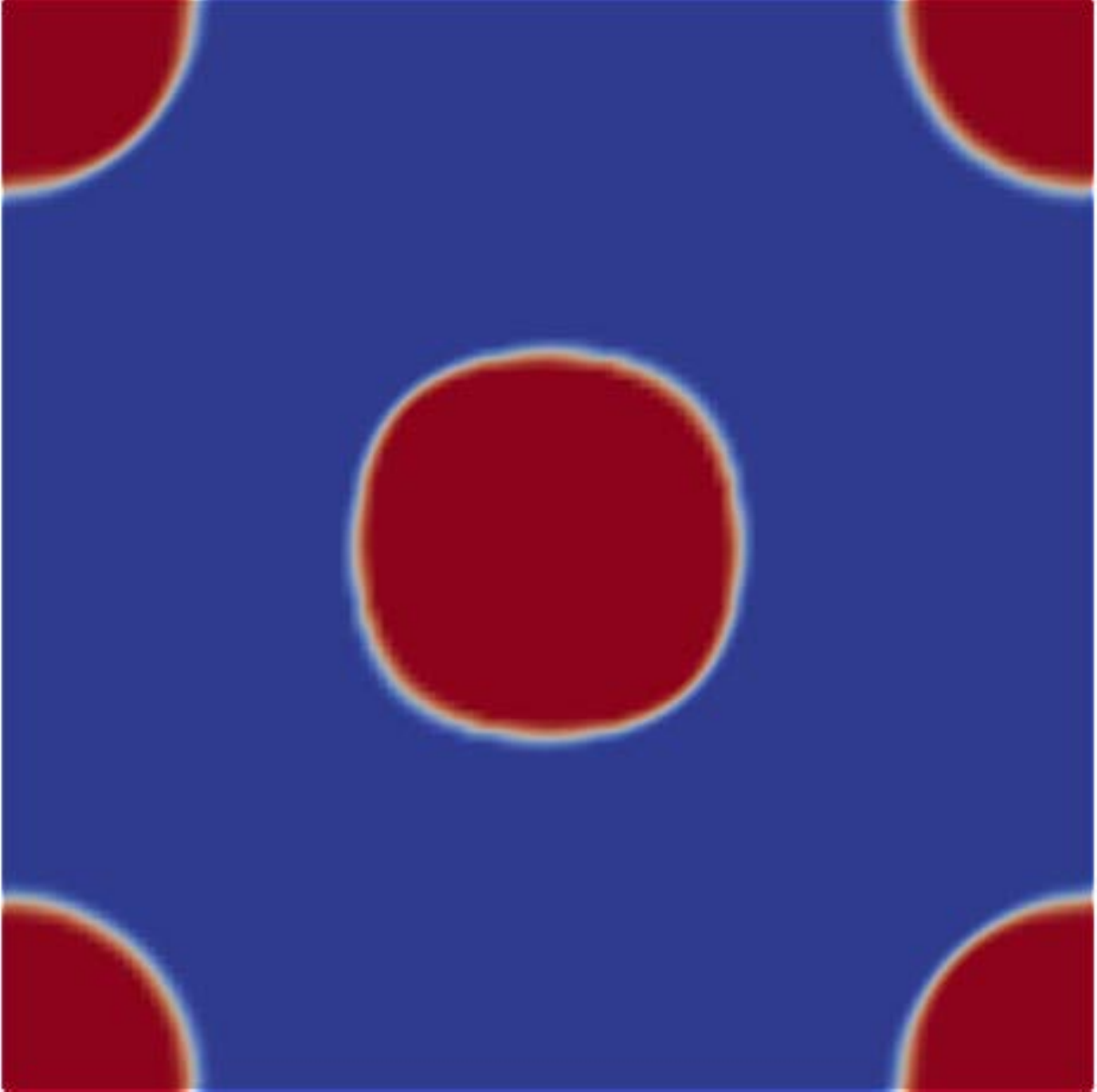}
 \caption{$m/|\Omega|=0.2$}
\end{subfigure}
\begin{subfigure}[b]{.2\textwidth}
 \includegraphics[width=\linewidth]{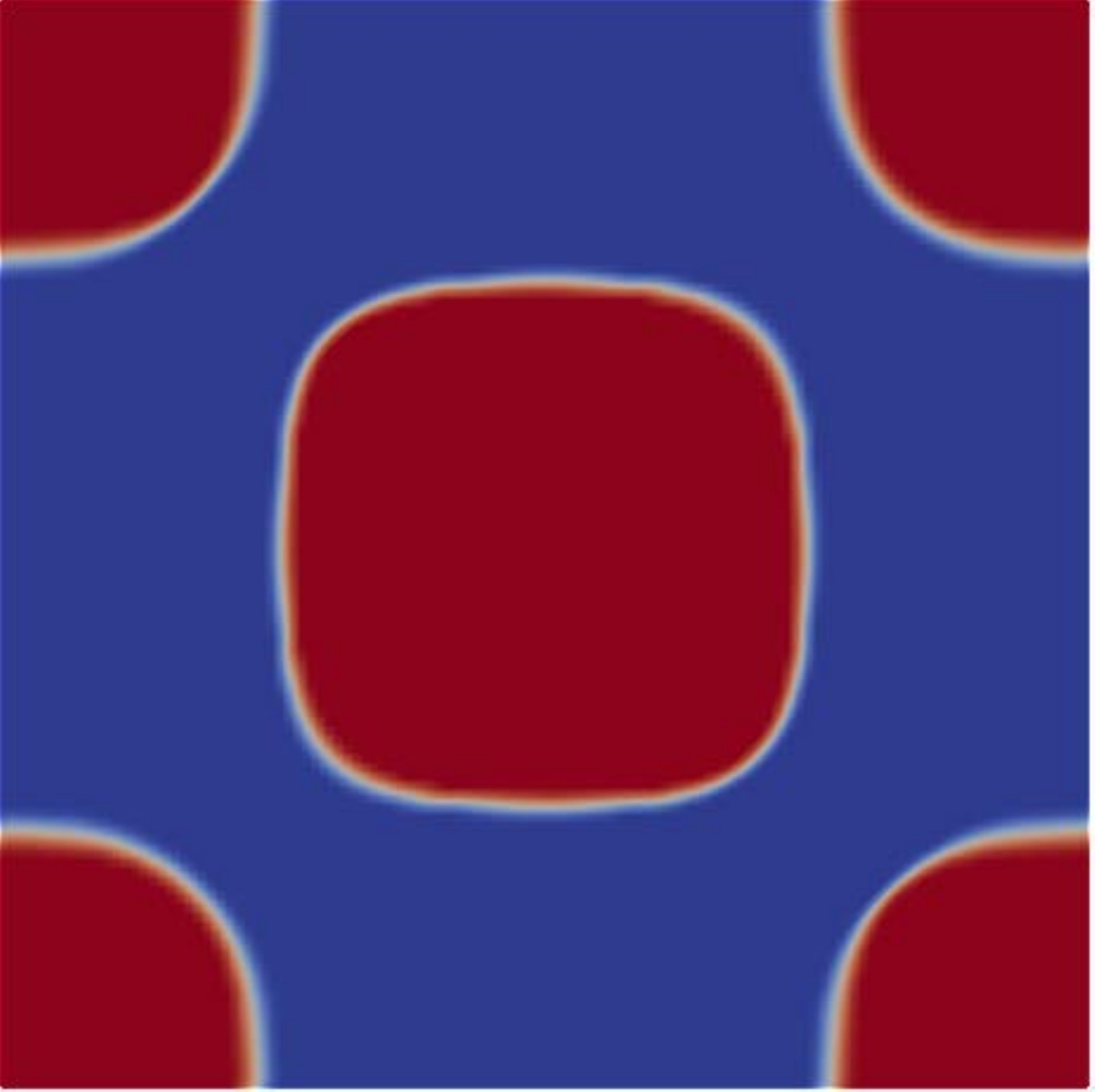}
 \caption{$m/|\Omega|=0.4$}
\end{subfigure}
\begin{subfigure}[b]{.2\textwidth}
 \includegraphics[width=\linewidth]{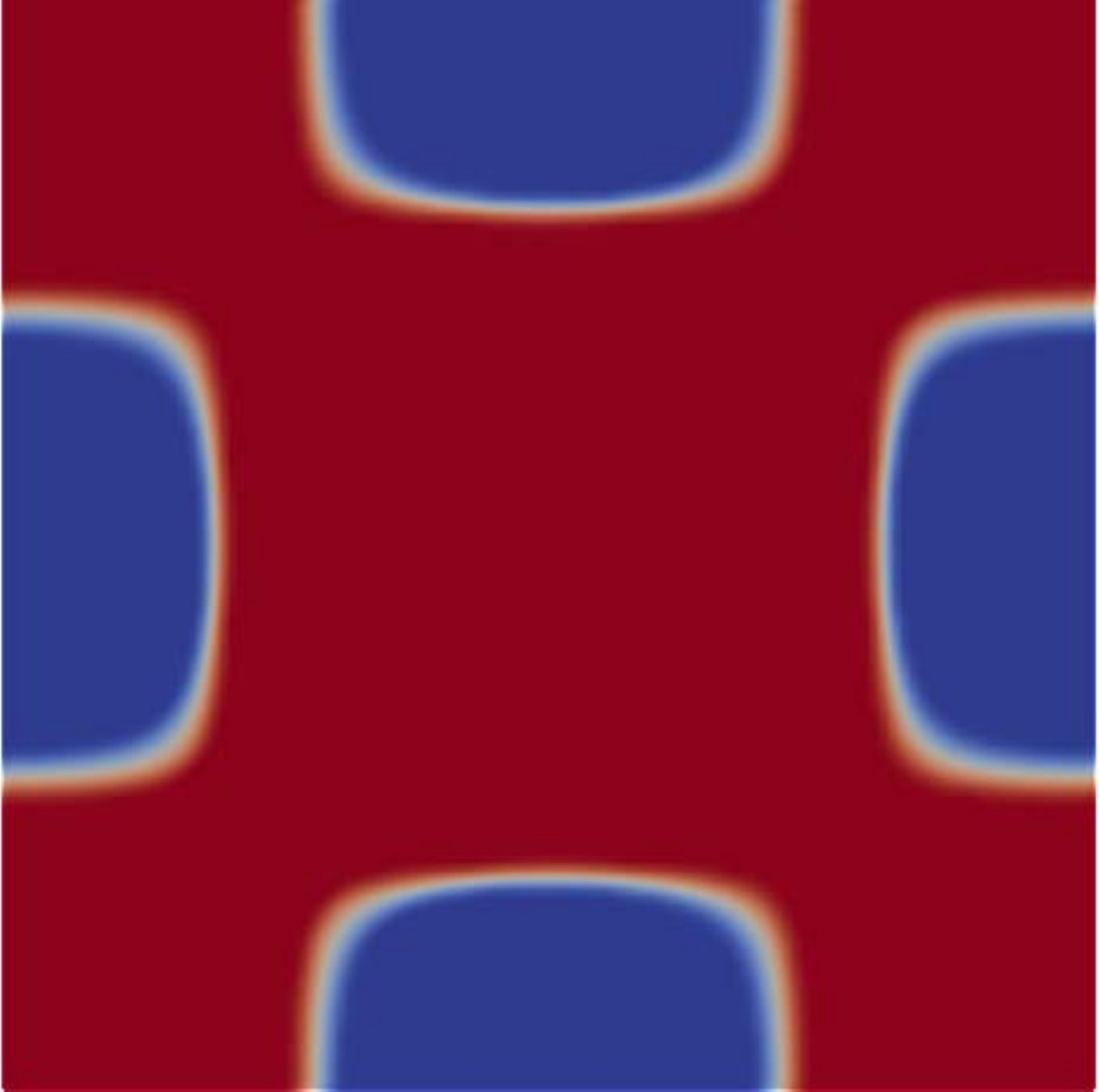}
 \caption{$m/|\Omega|=0.7$}
\end{subfigure}
\begin{subfigure}[b]{.2\textwidth}
 \includegraphics[width=\linewidth]{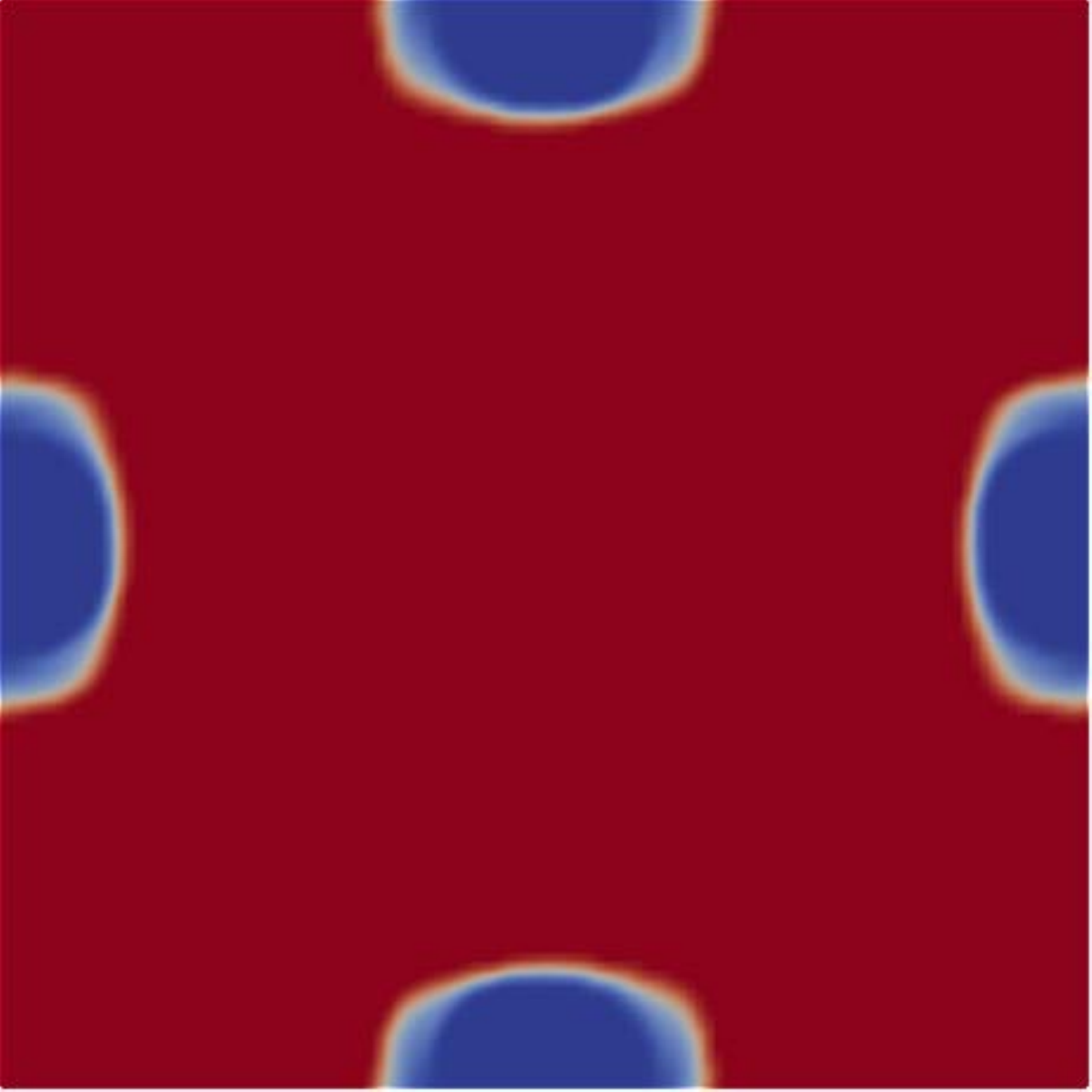}
 \caption{$m/|\Omega|=0.9$}
\end{subfigure}
\begin{subfigure}[b]{0.07\textwidth}
 \vfil
 \includegraphics[width=\linewidth]{colorbar}
 \\~
\end{subfigure}
\caption{Nearly optimal distribution $B$ in the square case for $\varepsilon=0.1$ .}
\label{square=e-1}
\end{figure}

Let us now present simulations on the unit cube $[0,1]^3$. For the visualisation, we have remove the phase where $\theta=0$. Since the computation have been made on a Laptop, the resolution is coarser in these simulations in dimension three, we kept the same numbers of degree of freedom.

\begin{figure}[!htH]

\centering
\begin{subfigure}[b]{.2\textwidth}
 \includegraphics[width=\linewidth]{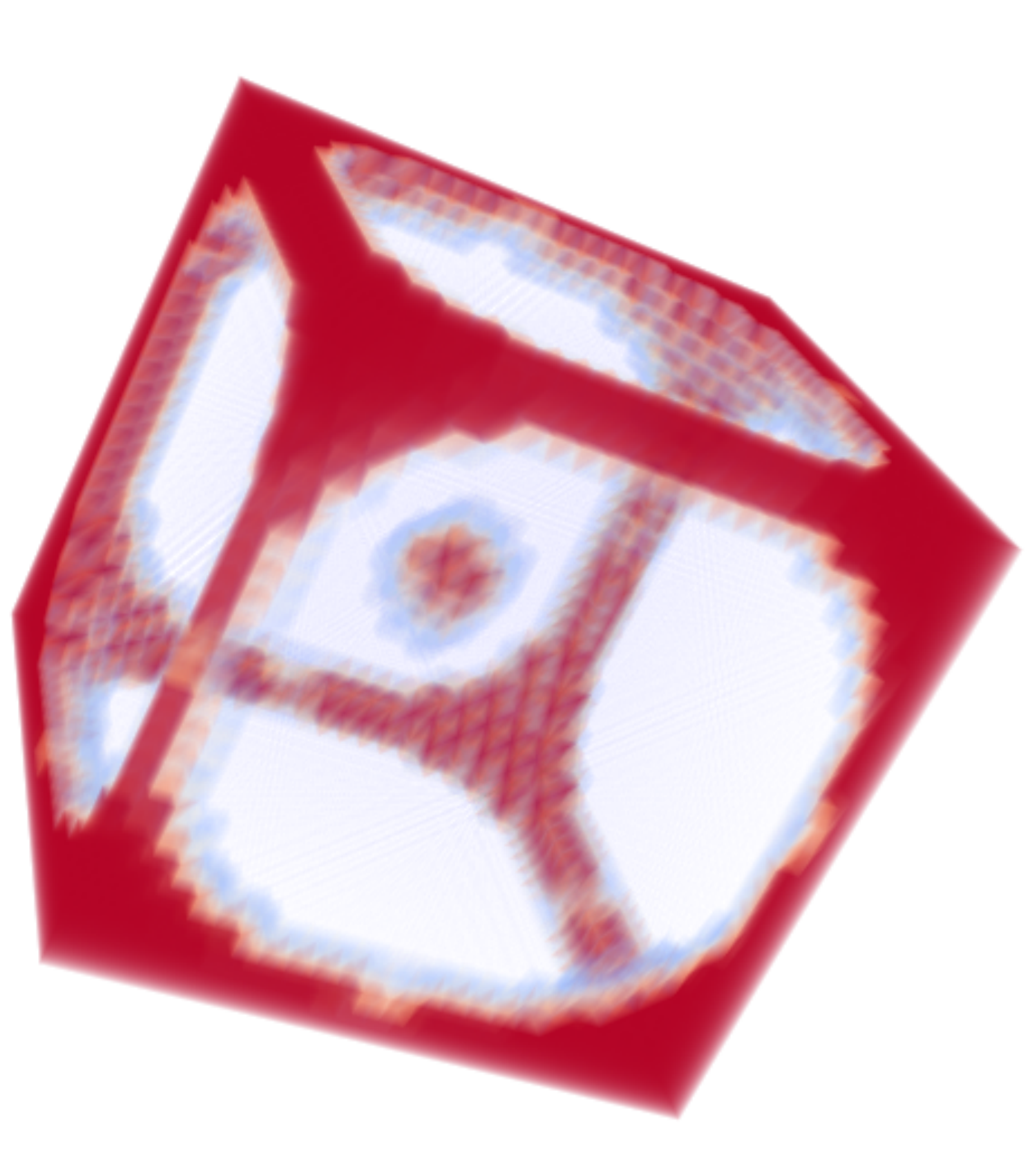}
 \caption{$m/|\Omega|=0.125$}
\end{subfigure}
\begin{subfigure}[b]{.2\textwidth}
 \includegraphics[width=\linewidth]{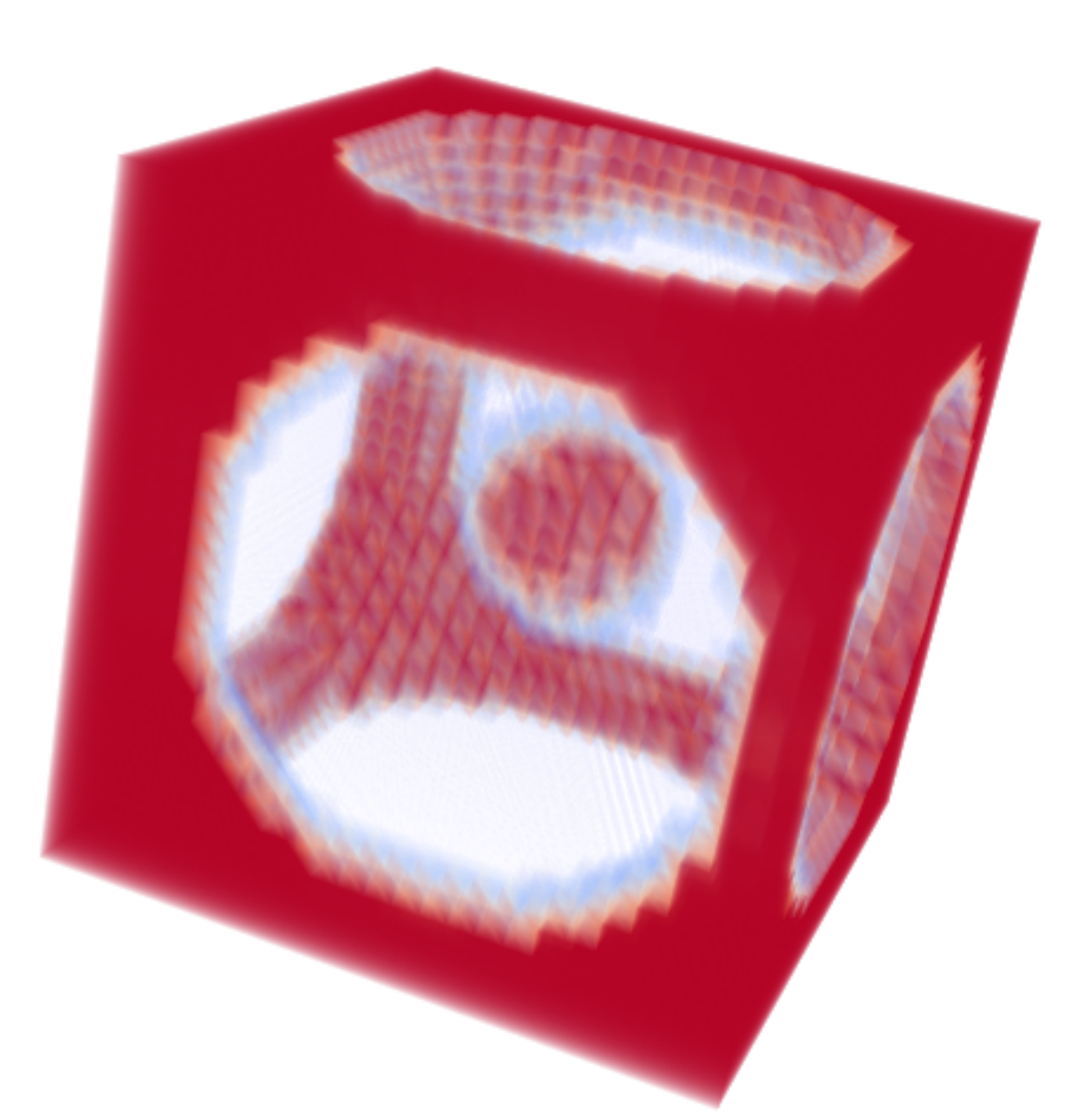}
 \caption{$m/|\Omega|=0.25$}
\end{subfigure}
\begin{subfigure}[b]{.2\textwidth}
 \includegraphics[width=\linewidth]{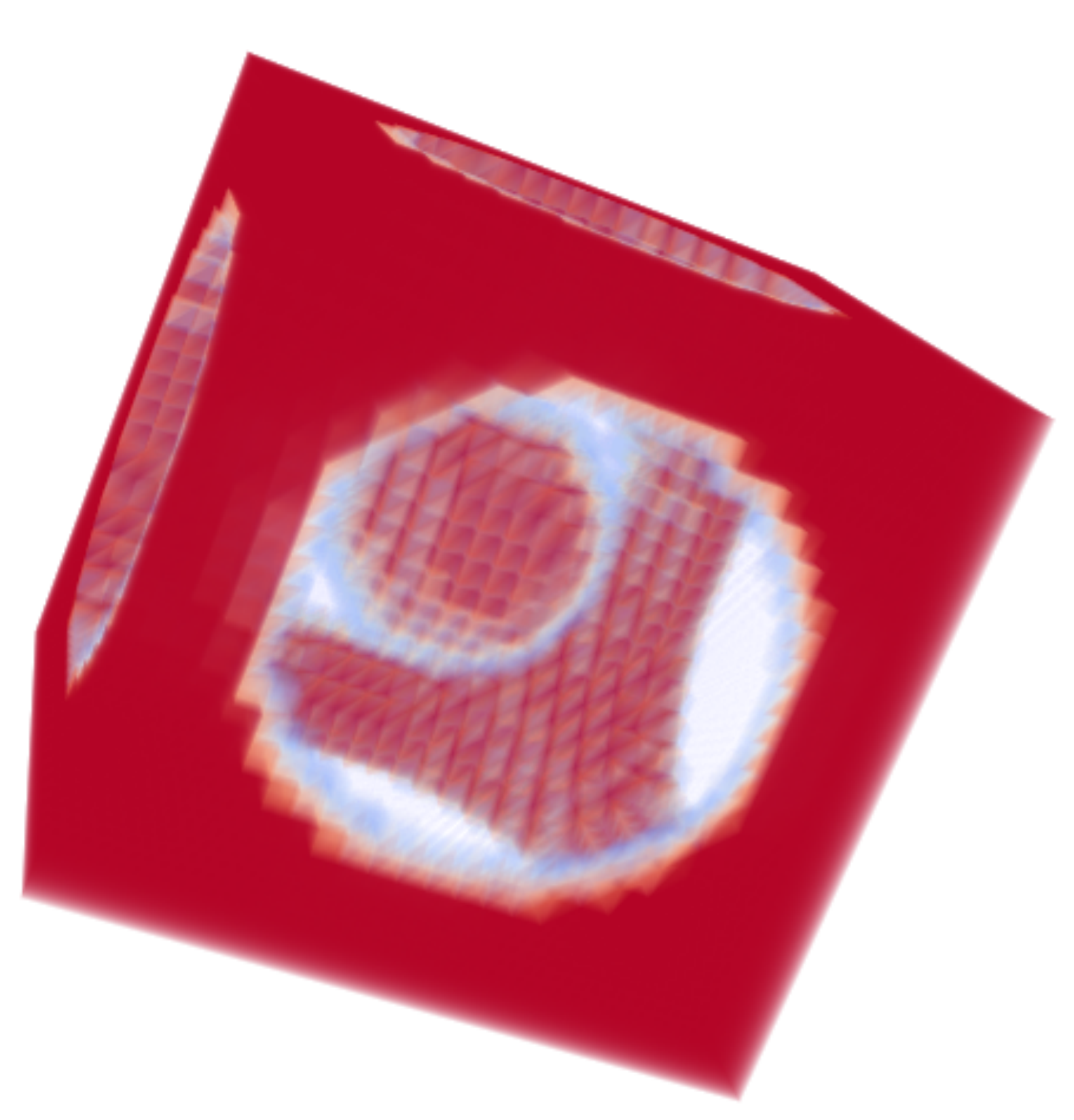}
 \caption{$m/|\Omega|=0.375$}
\end{subfigure}
\begin{subfigure}[b]{.2\textwidth}
 \includegraphics[width=\linewidth]{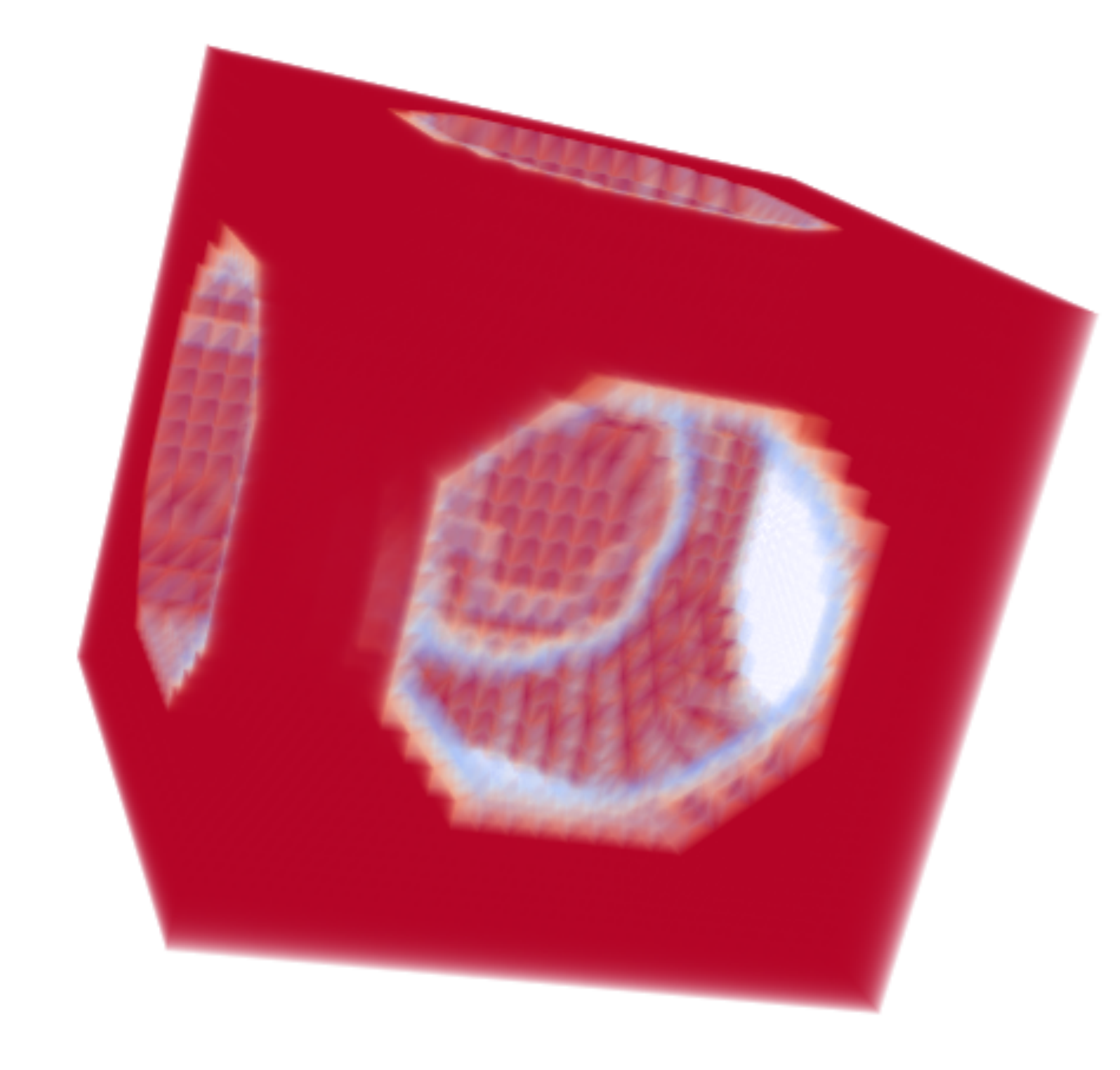}
 \caption{$m/|\Omega|=0.5$}
\end{subfigure}
\begin{subfigure}[b]{0.07\textwidth}
 \vfil
 \includegraphics[width=\linewidth]{colorbar}
 \\~
\end{subfigure}
\caption{Nearly optimal distribution $B$ in the cube case for $\varepsilon=10^{-6}$. }
\label{cube}

\end{figure}

\subsection{Others domains}

For the sake of completeness, we present computations in other plane domains for the comparison with \cite{ConcaLaurainMahadevan}: a crescent in Figure \ref{crescent=e-6} and a perforated ellipse in Figure \ref{ellipse=e-6}.

\begin{figure}[!htH]
\centering
\begin{subfigure}[b]{.2\textwidth}
 \includegraphics[width=\linewidth]{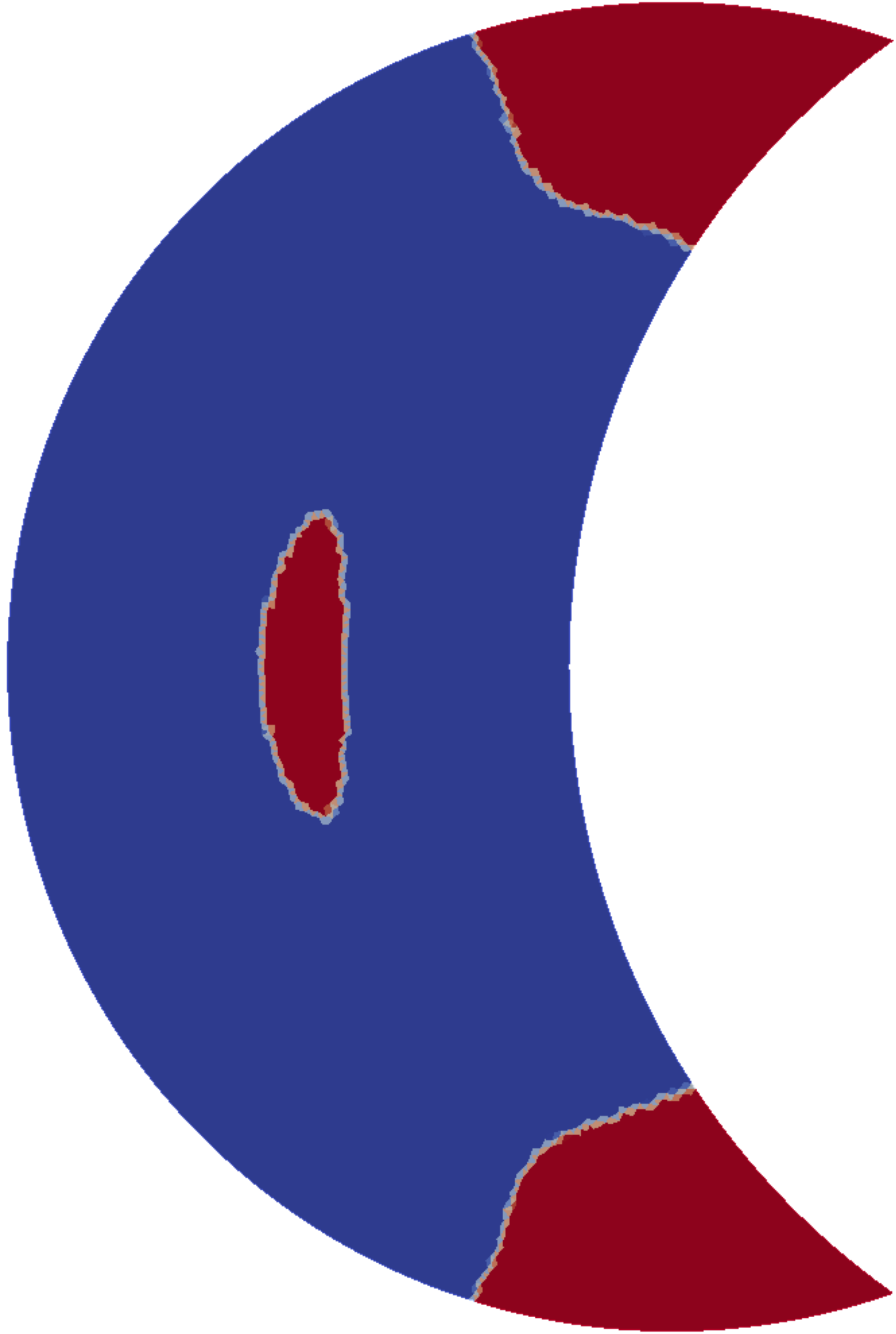}
 \caption{$m/|\Omega|=0.2$}
\end{subfigure}
\begin{subfigure}[b]{.2\textwidth}
 \includegraphics[width=\linewidth]{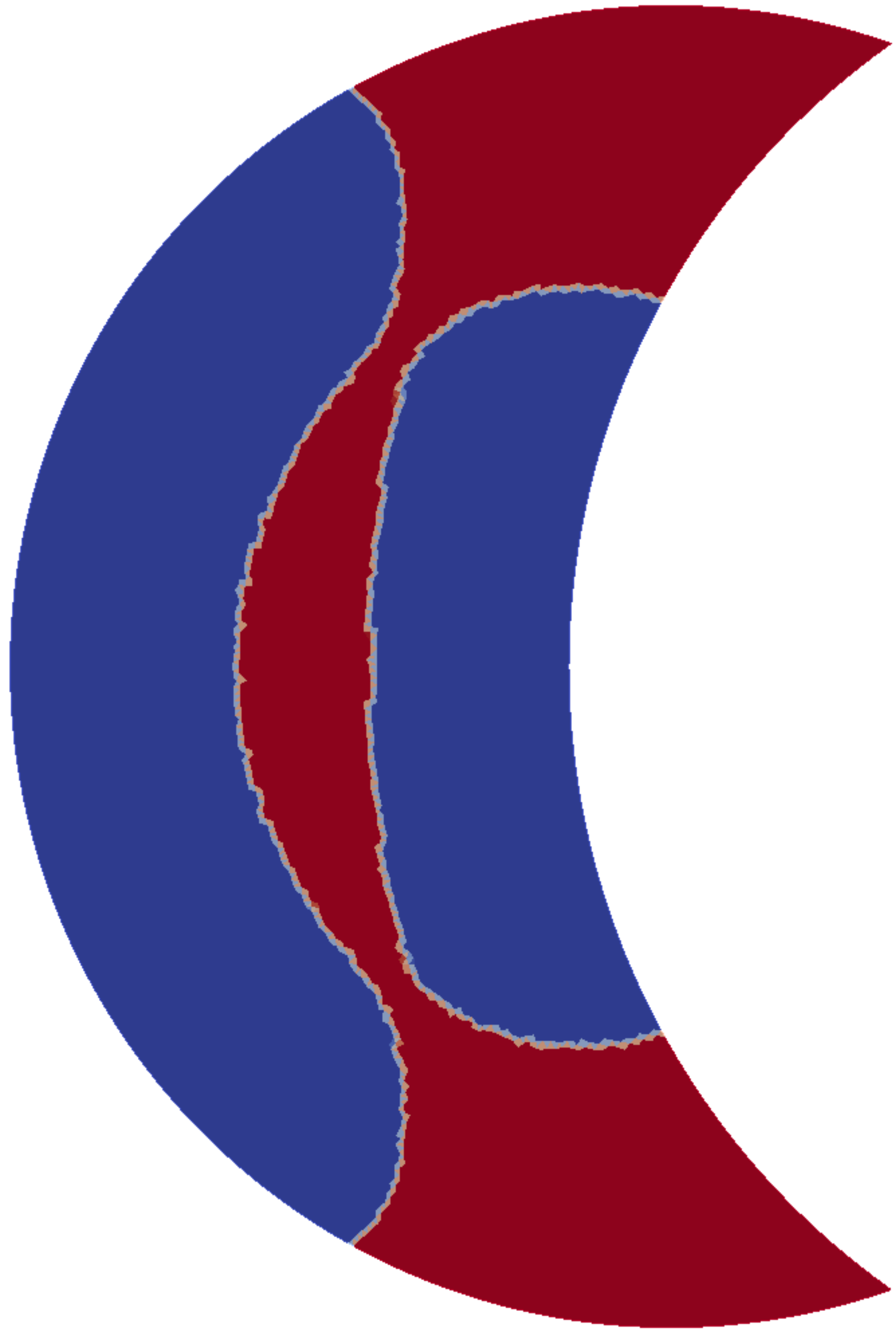}
 \caption{$m/|\Omega|=0.4$}
\end{subfigure}
\begin{subfigure}[b]{.2\textwidth}
 \includegraphics[width=\linewidth]{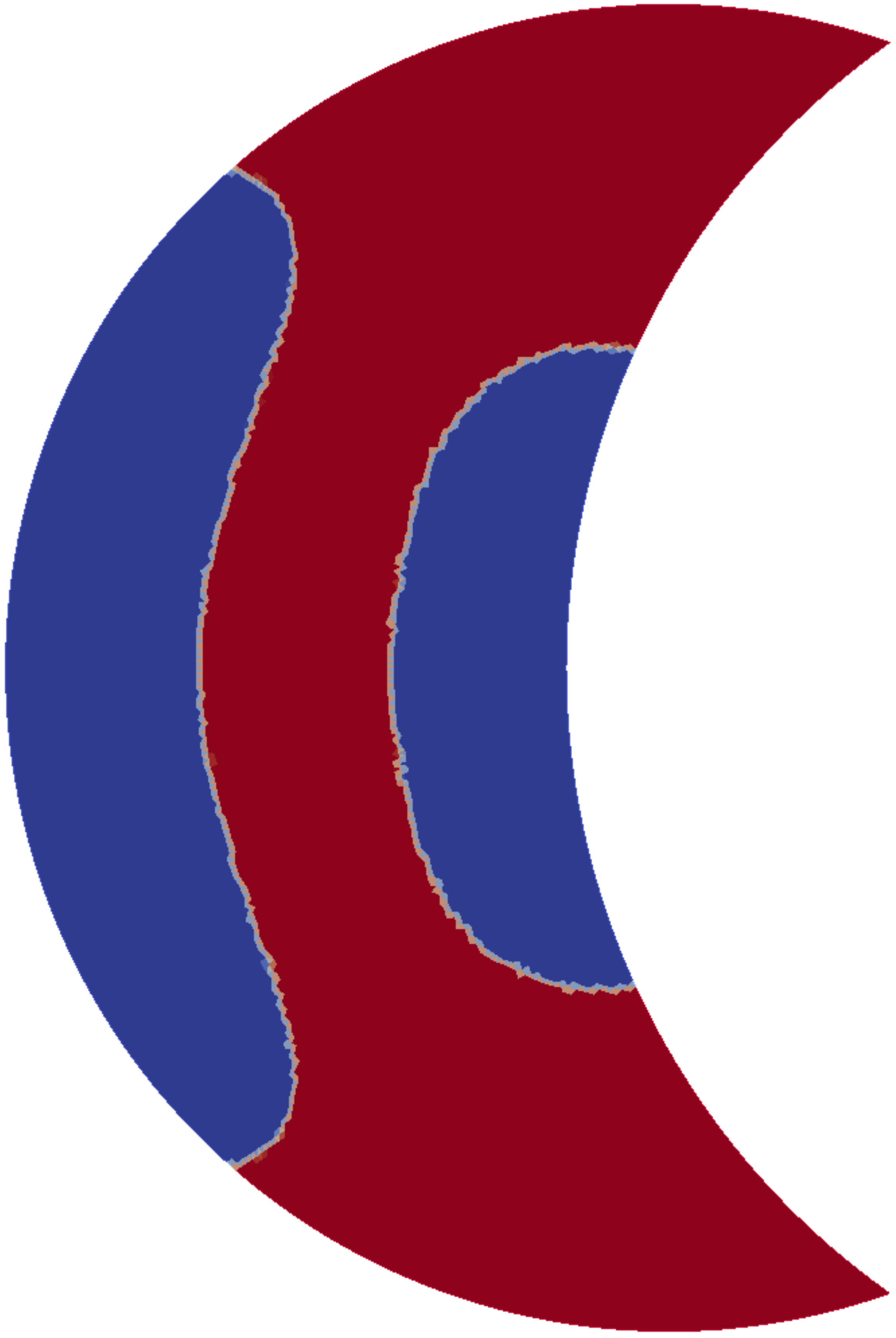}
 \caption{$m/|\Omega|=0.6$}
\end{subfigure}
\begin{subfigure}[b]{.2\textwidth}
 \includegraphics[width=\linewidth]{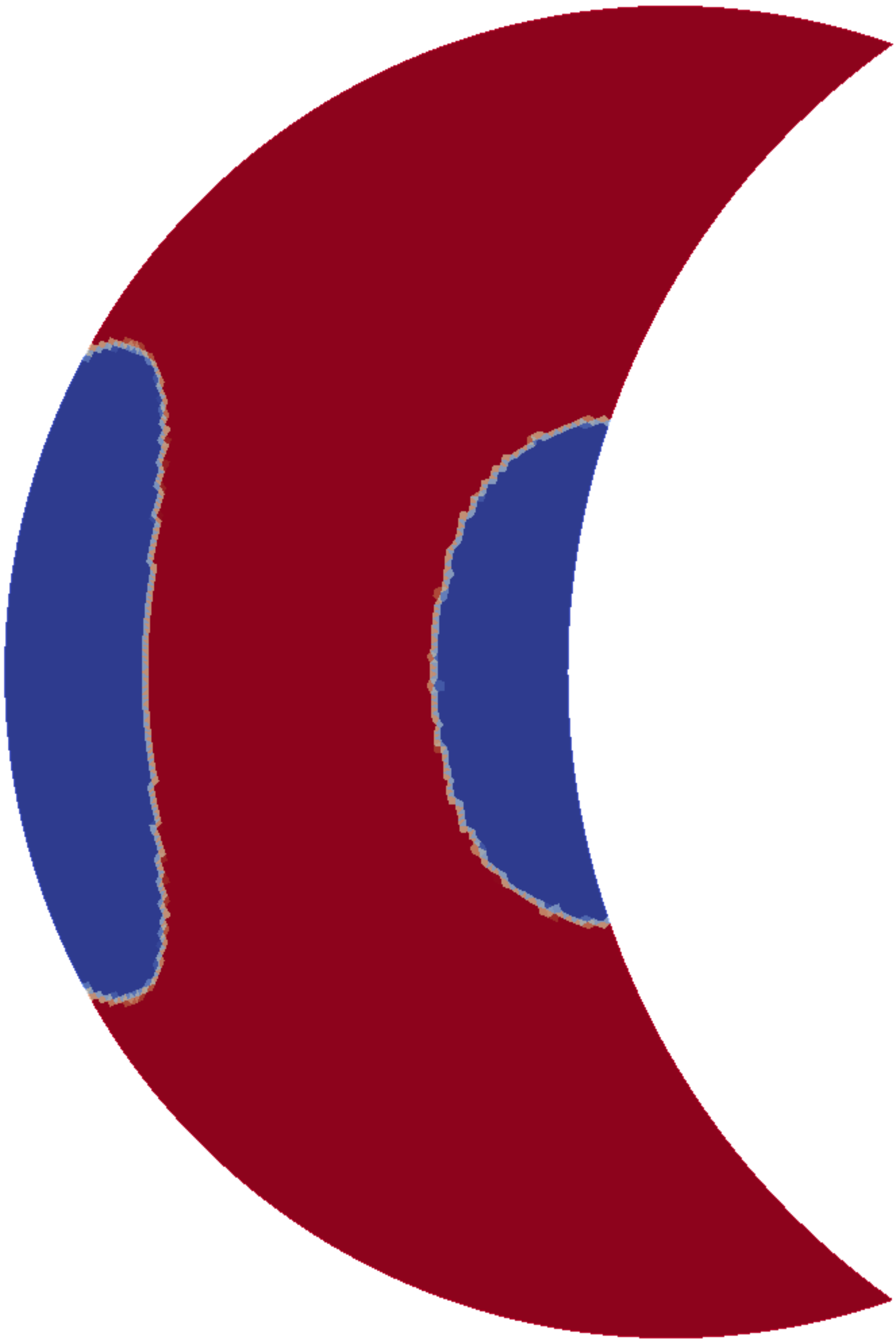}
 \caption{$m/|\Omega|=0.8$}
\end{subfigure}
\begin{subfigure}[b]{0.07\textwidth}
 \vfil
 \includegraphics[width=\linewidth]{colorbar}
 \\~\\~\\~
\end{subfigure}
\caption{Nearly optimal distribution $B$ in a crescent for $\varepsilon=10^{-6}$ .}
\label{crescent=e-6}
\end{figure}

\begin{figure}[!htH]
\centering
\begin{subfigure}[b]{.3\textwidth}
 \includegraphics[width=\linewidth]{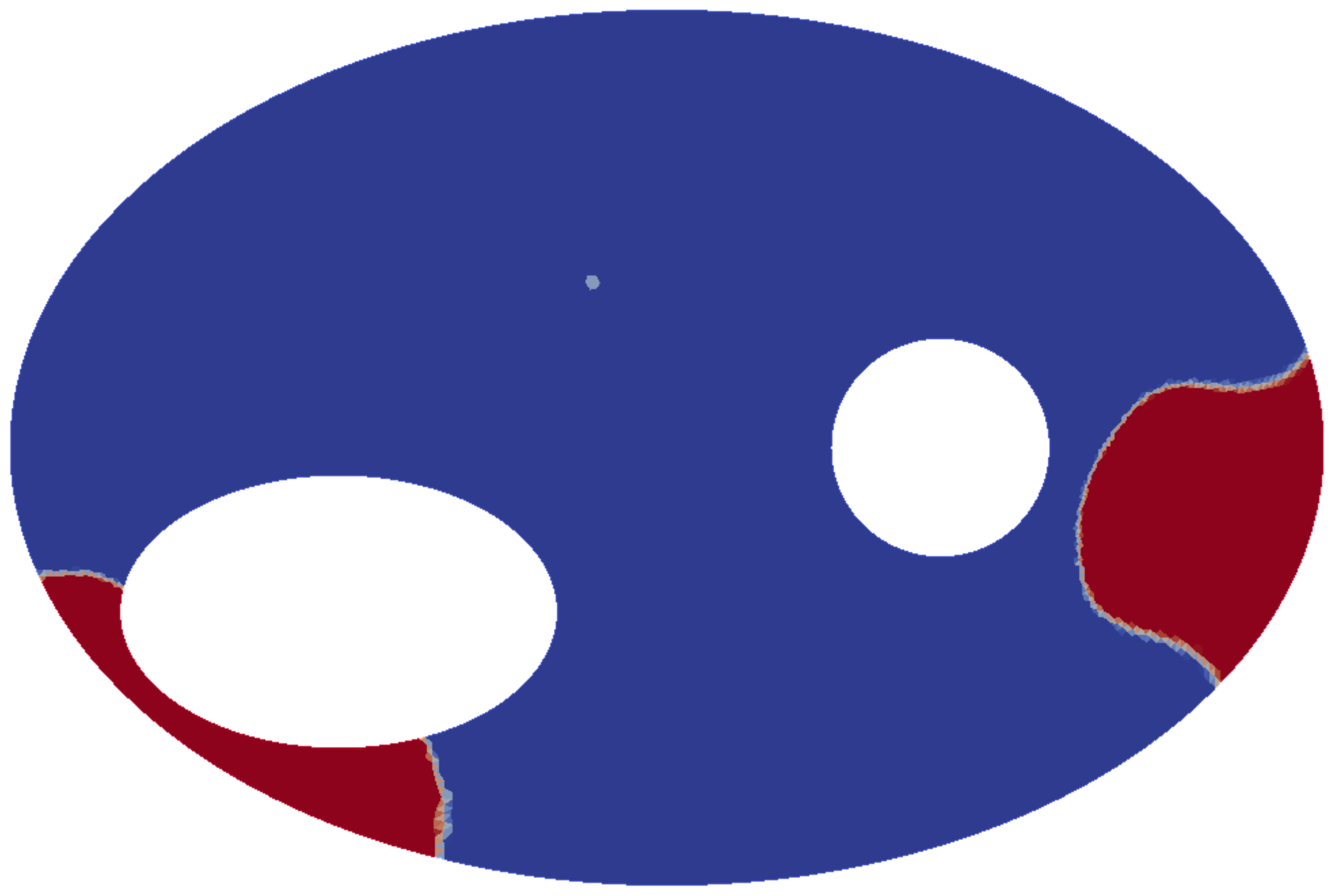}
 \caption{$m/|\Omega|=0.1$}
\end{subfigure}
\begin{subfigure}[b]{.3\textwidth}
 \includegraphics[width=\linewidth]{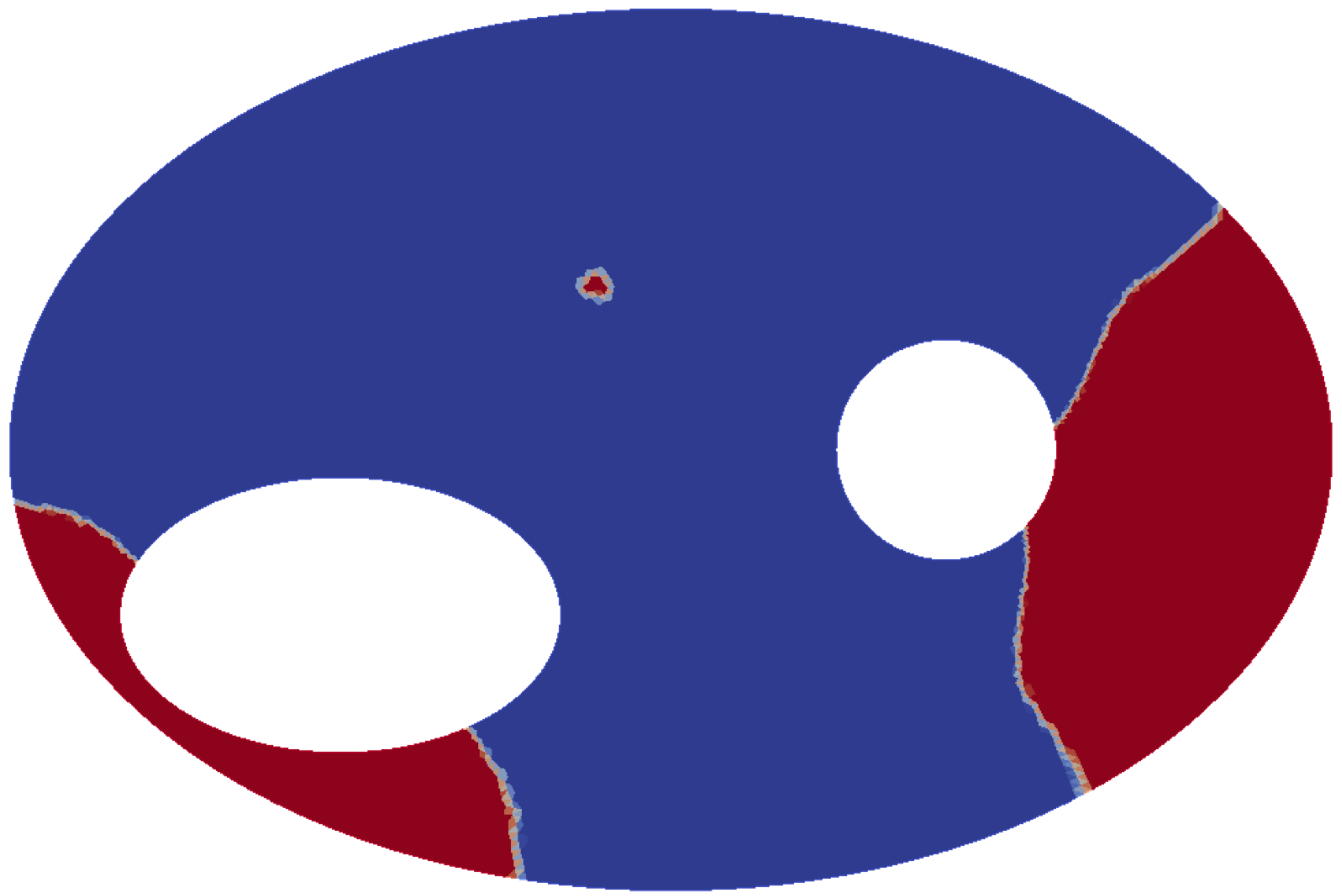}
 \caption{$m/|\Omega|=0.2$}
\end{subfigure}
\begin{subfigure}[b]{.3\textwidth}
 \includegraphics[width=\linewidth]{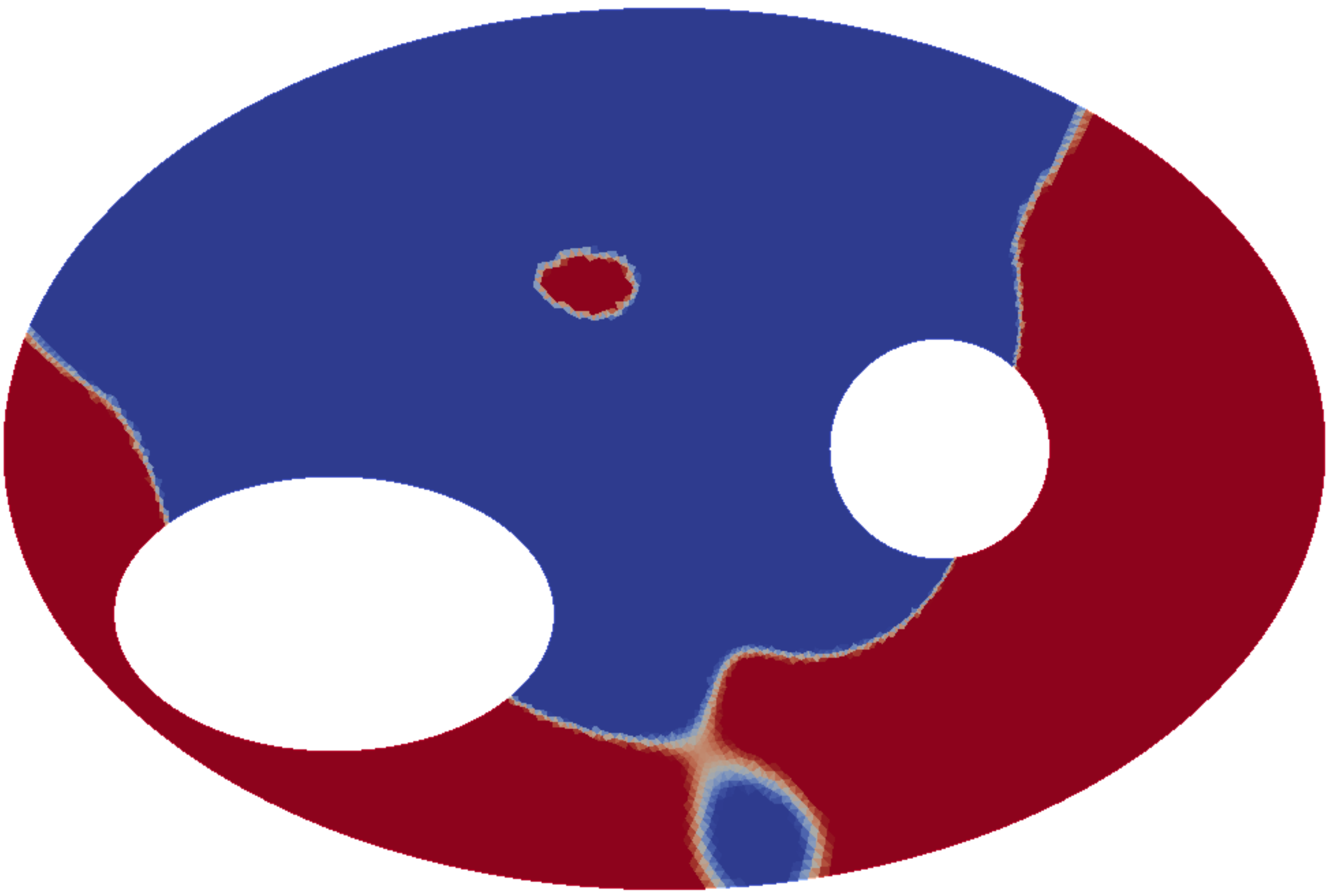}
 \caption{$m/|\Omega|=0.4$}
\end{subfigure}

\begin{subfigure}[b]{.3\textwidth}
 \includegraphics[width=\linewidth]{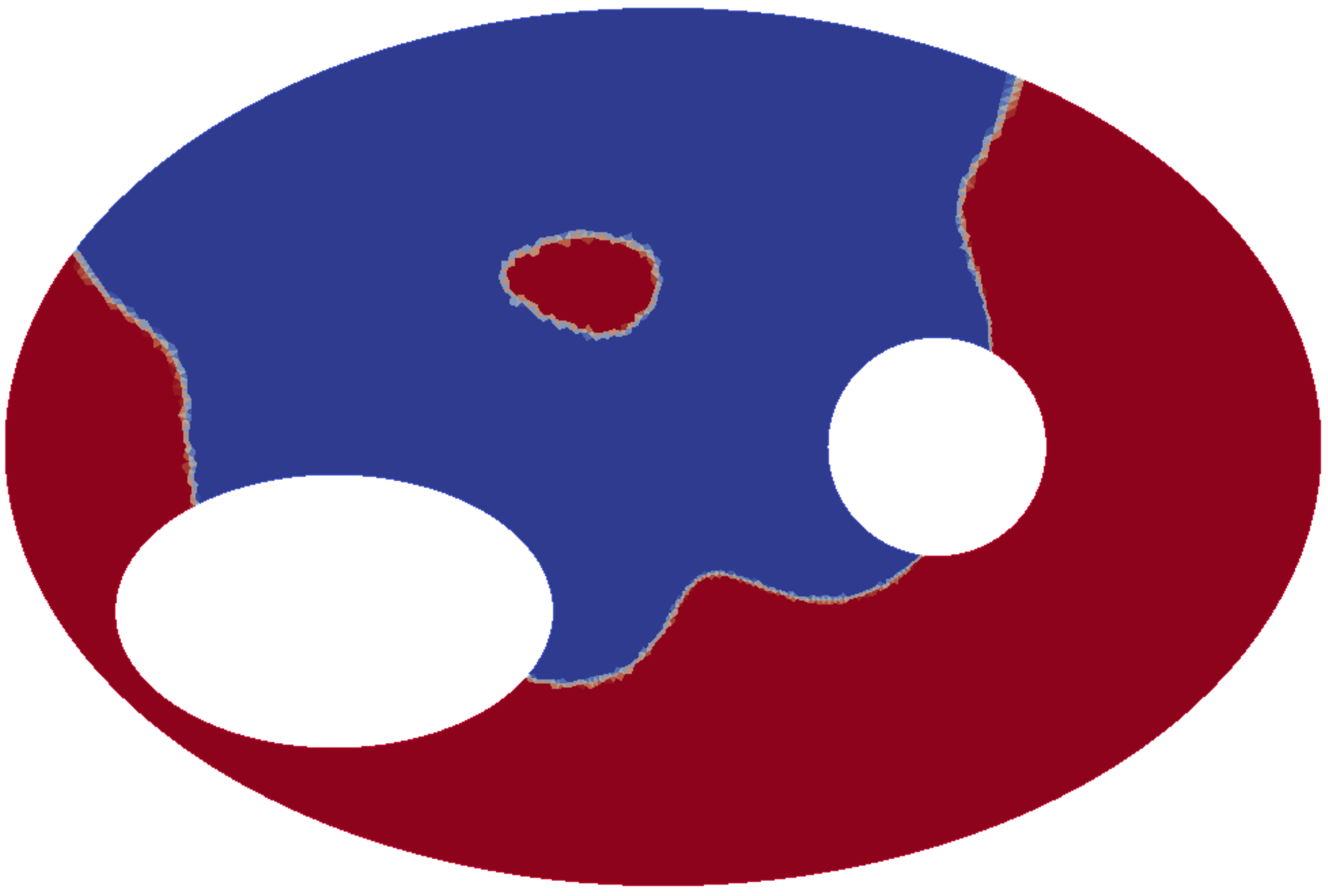}
 \caption{$m/|\Omega|=0.5$}
\end{subfigure}
\begin{subfigure}[b]{.3\textwidth}
 \includegraphics[width=\linewidth]{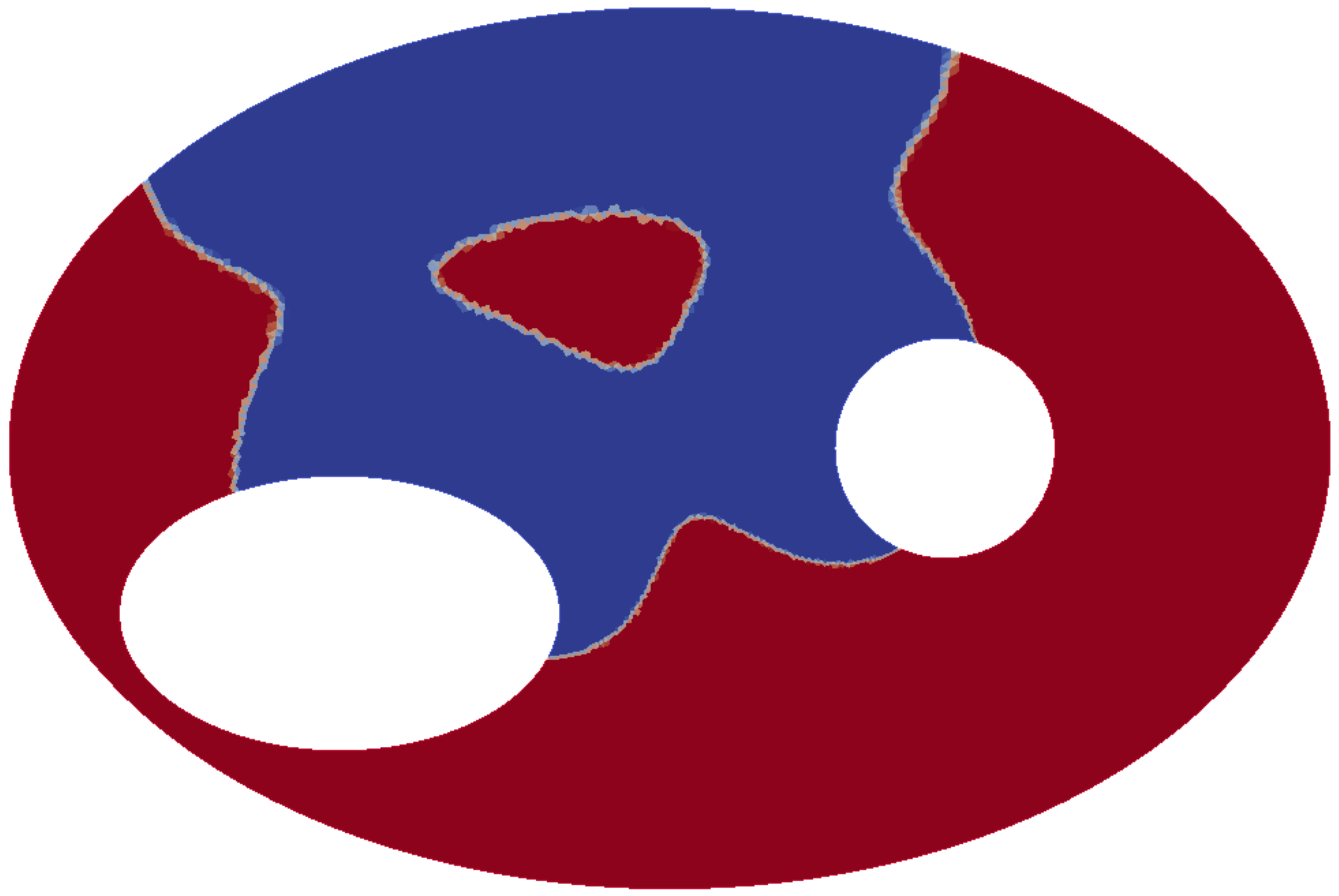}
 \caption{$m/|\Omega|=0.6$}
\end{subfigure}
\begin{subfigure}[b]{.3\textwidth}
 \includegraphics[width=\linewidth]{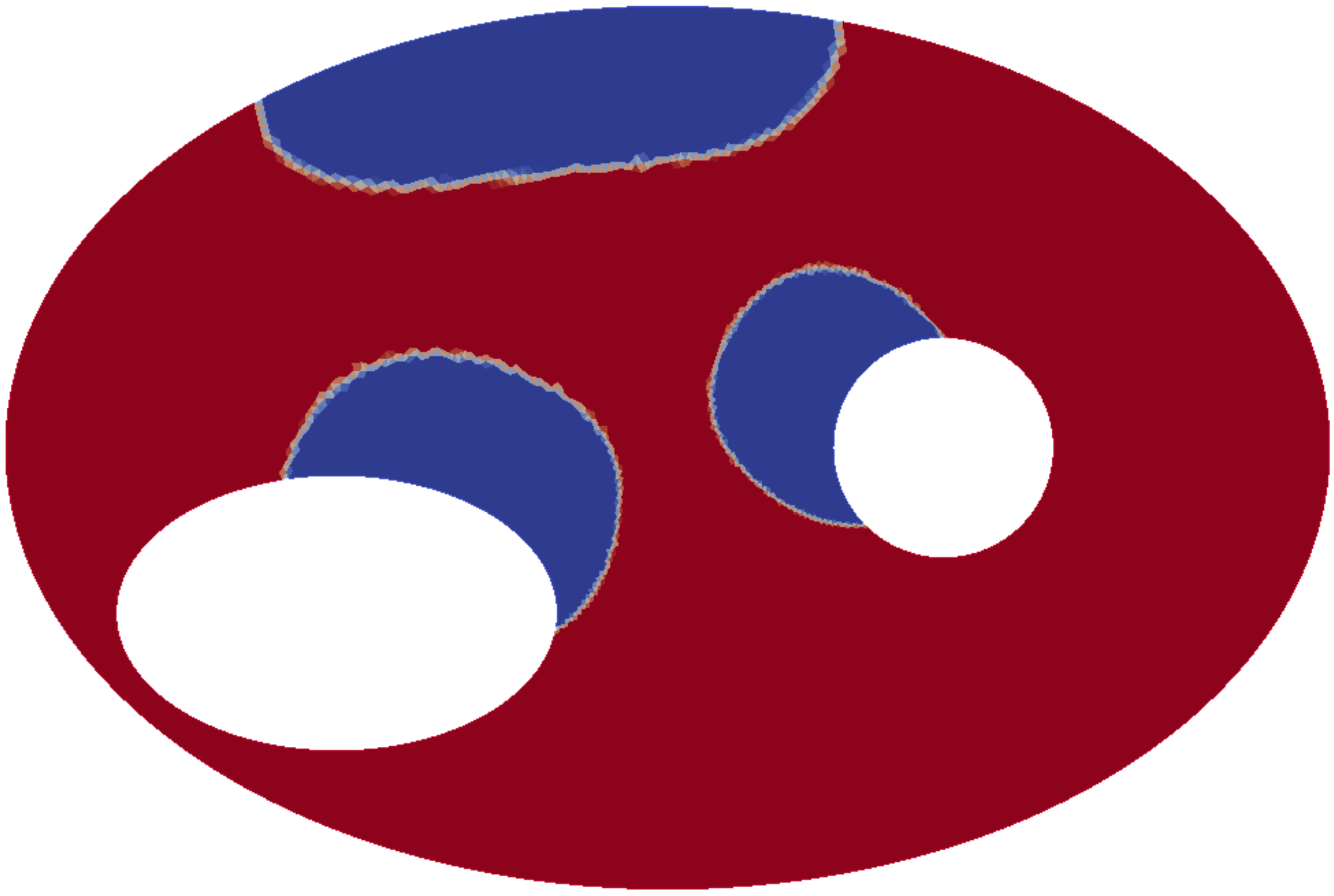}
 \caption{$m/|\Omega|=0.8$}
\end{subfigure}
\caption{Nearly optimal distribution $B$ in a perforated ellipse for $\varepsilon=10^{-6}$ .}
\label{ellipse=e-6}
\end{figure}

Let us emphasize that in the last case, even for $\varepsilon=10^{-6}$, we observe clearly in Figure \ref{ellipse=e-6}(c) a small area where $\theta$ takes values strictly between $0$ and $1$ where we see the effect of the modelling with a second order approximation.

%%%%%%%%%%%%%%%%%%%%%%%%%%%%%%%%%%%%%%%%%%%
%%%% \input{sec4_r}
\section{Appendix}
\setcounter{equation}{0}
\setcounter{theorem}{0}

\begin{lemma}
\label{pseudo:ordre0}
For $\theta \in L^\infty(\Omega)$, if $z \in H^1_0(\Omega)$ solves 
\begin{equation}\label{eqz}
-\Delta z=\dive(\theta \nabla u_0) \quad \mbox{ in  }\R^N\,,
\end{equation}
 then $\theta \mapsto Q(\theta): =\nabla z$ defines  a pseudo-differential operator with symbol
\begin{equation}\label{symbol}
q(x,\xi)=-\,\frac{\xi\cdot\nabla u_0(x)}{|\xi|^2}\,\xi.
\end{equation}
Note that $q$ is homogenous of degree $0$ in $\xi$.
\end{lemma}

\begin{proofof}{Lemma \ref{pseudo:ordre0}}
We first consider the whole space case in order to use Fourier calculus. Indeed, denoting by $\,\widehat{}\,$ the Fourier transform and starting from Equation (\ref{eqzn}), formally we can calculate as follows
\begin{align*}
  (-\Delta z)\,\widehat{}\ (\xi) &= (\dive(\theta\nabla u_0))\,\widehat{}\ (\xi)\\
  -(-|\xi|^2\widehat{z}) &= -i\xi\cdot\nabla u_0(x)\widehat{\theta}\\
  \widehat{z} &= -\frac{i\xi\cdot\nabla u_0(x)}{|\xi|^2}\,\widehat{\theta},
\end{align*}
which gives
\begin{equation*}
  \widehat{\nabla z}(\xi) = -i\xi\widehat{z}(\xi) = -i\xi\left(-\frac{i\xi\cdot\nabla u_0(x)}{|\xi|^2}\,\widehat{\theta}\right)= \underbrace{\left(-\frac{\xi\cdot\nabla u_0(x)}{|\xi|^2}\,\xi\right)}_{q(x,\xi)}\widehat{\theta}.
\end{equation*}
\end{proofof}

\paragraph{Acknowledgements:} The authors thank support from ECOS-CONICYT Grant C13 05. The first author is also partially supported by Fondecyt Grant N$^\circ$ 1140773. The second author is also partially supported by the ANR Grant ARAMIS and OPTIFORM. The last author acknowledges the support of Fondecyt Grant N$^\circ$  1130595. 

\bibliographystyle{plain}

\bibliography{CDMQ2013_r}

\noindent------------------------------------------------------------------

\noindent Carlos Conca

\smallskip

\noindent Universidad de Chile
\newline Department of Engineering Mathematics,
Center for Mathematical Modelling (CMM), UMI 2807 CNRS-Chile
\& Center for Biotechnology and Bioengineering (CeBiB),
Universidad de Chile, Santiago, Chile.
% \newline Avenue de l'Universit\'{e} - BP 1155, 64013 Pau, France.

\smallskip

\noindent E-mail: \texttt{cconca@dim.uchile.cl}
\newline\bigskip

\noindent Marc Dambrine

\smallskip

\noindent D\'epartement de Math\'ematiques, Universit\'{e} de Pau et des Pays de l'Adour
% \newline Avenue de l'Universit\'{e} - BP 1155, 64013 Pau, France.

\smallskip

\noindent E-mail: \texttt{marc.dambrine@univ-pau.fr}
\newline\bigskip

\noindent Rajesh Mahadevan

\smallskip

\noindent Depto. de Matem\'atica, Fac. Cs. F\'is. y Matem\'aticas, Universidad de Concepci\'on

\smallskip

\noindent E-mail: \texttt{rmahadevan@udec.cl}
\newline\bigskip

\noindent Duver Quintero

\smallskip

\noindent Universit\'e de Pau et des Pays de l'Adour \& Universidad de Chile.

\smallskip

\noindent E-mail: \texttt{duver@dim.uchile.cl}
\newline\bigskip

\end{document}